\documentclass[11pt]{amsart}
\usepackage{graphicx}  %插入图片的宏包
\usepackage{subfig}  %插入多图时用子图显示的宏包
\usepackage{amssymb, amsmath, amsthm}
\usepackage[colorlinks=true,linkcolor=blue,citecolor=red]{hyperref}
\hypersetup{hypertexnames=false}
\usepackage{color}
\usepackage{float}
\usepackage{enumerate}
\usepackage{graphics}
\usepackage{epstopdf}
\usepackage{tikz}
\usetikzlibrary{arrows,decorations.markings}
\usepackage{cite}
\usepackage{epsfig}
\usepackage{amsmath}
\usepackage{amssymb}
\usepackage{amscd}
\usepackage{graphicx}
\usepackage{mathrsfs}
\usepackage{needspace}
\usepackage[T1]{fontenc}
\usepackage{bm,bbm}
\hypersetup{hypertexnames=false}

\makeatletter
\newcommand{\equationlabels}[1]{%
  \begingroup
  \edef\@currentlabel{\theequation}%
  \edef\@currentHref{equation.\theHequation}%
  #1%
  \endgroup}
\makeatother
\numberwithin{equation}{section}
\newtheorem{theorem}{Theorem}[section]
\newtheorem{corollary}[theorem]{Corollary}
\newtheorem{lemma}[theorem]{Lemma}
\newtheorem{proposition}[theorem]{Proposition}

\newtheorem{dbar-RHP}[theorem]{$\bar{\partial}$-RH problem}
\theoremstyle{definition}
\newtheorem{definition}[theorem]{Definition}
\newtheorem{remark}[theorem]{Remark}

\newcommand{\R}{\mathbb{R}}
\newcommand{\C}{\mathbb{C}}
\newcommand{\N}{\mathbb{N}}

\newcommand{\uu}{\mathbf{u}}
\newcommand{\Ub}{\mathbf{U}}
\newcommand{\Rb}{\mathbf{R}}
\newcommand{\Ss}{\mathbb{S}}
\newcommand{\sv}{\mathbf{s}}

\newcommand{\Qv}{\mathbf{Q}}
\newcommand{\qv}{\mathbf{q}}
\newcommand{\ov}{\overline}
\newcommand{\ran}{\mathrm{ran}}
\renewcommand{\ker}{\mathrm{ker}}
\newcommand{\dom}{\mathrm{dom}}
\newcommand{\rank}{\mathrm{rank}}
\newcommand{\ii}{\mathrm{i}}
\newcommand{\eu}{\mathrm{e}}
\newcommand{\Gr}{\mathsf{Gr}}
\newcommand{\id}{\mathrm{Id}}
\newcommand{\Tr}{\mathrm{Tr}}

\newcommand{\Hfr}{\mathfrak{H}}



\begin{document}

\title[Sharp well-posedness and soliton resolution]{Sharp well-posedness and soliton resolution for the half-wave maps equation in $H^{1/2}_\bullet$}

%\today

\author[Tang]{Fei Tang}
\author[Tian]{Shou-Fu Tian$^{*,\dagger}$}
\address{Fei Tang  \newline
School of Mathematics, China University of Mining and Technology, Xuzhou 221116, People's Republic of China}
\email{ftang@cumt.edu.cn}
\address{Shou-Fu Tian (Corresponding author)\newline
School of Mathematics, China University of Mining and Technology, Xuzhou 221116, People's Republic of China}
\email{sftian@cumt.edu.cn}

\thanks{*Corresponding author (sftian@cumt.edu.cn).}
\thanks{$^\dagger$Contributed equally as the first author.}

\begin{abstract}
  We prove sharp well-posedness and soliton resolution for the half-wave maps equation on the real line $\R$ in \(H^{1/2}_\bullet\). This regularity threshold for global well-posedness is optimal, as we show by constructing examples of ill-posedness in \(H^s_\bullet\) for any \(s<1/2\). Moreover, with the spectral assumption satisfied by a dense set of initial data in the energy space, we prove that solutions in \(H^{1/2}_\bullet\) asymptotically decouple as an infinite superposition of solitons. Our results extend the known global well-posedness and finite soliton resolution for rational data in \(H^\infty_\bullet\) to sharp global well-posedness and soliton resolution for general data in \(H^{1/2}_\bullet\).

  Different from rational solutions, general data can also involve an infinite discrete spectrum, which makes the reduction analysis framework on the finite rank no longer applicable. For sharp well-posedness, we present a new isometry criterion and stability results for resolvent transform in the explicit formula. For soliton resolution, we develop a spectral method for this equation based on direct resolvent analysis and a novel trace norm estimate that upgrades energy convergence to uniform convergence. Our argument also yields sharp well-posedness and soliton resolution for the generalized half-wave maps equation with targets in the complex Grassmannians, extending the classical sphere-valued case.
\end{abstract}

	\maketitle

\section{Introduction}\label{sec.1}
\indent We investigate the Cauchy problem of the half-wave maps equation on the real line
\begin{align}\label{HWMv}\tag{HWM}
	\partial_t \uu = \uu \times |D| \uu,\enspace \uu:[0,T)\times\R \to\Ss^2.
\end{align}
Here $\mathbb{S}^2$ is the standard unit sphere in $\mathbb{R}^3$, and $\times$ in the equation is the cross product. The operator $|D| = \mathcal{H}\partial_x$ has the Fourier multiplier $|\xi|$, where the Hilbert transform $\mathcal{H}$ is given by
\begin{align*}
	(\mathcal{H}f)(x)=\frac{1}{\pi}\mathrm{p.v.}\int_{\R}\frac{f(y)}{x-y}dy.
\end{align*}
From a physical perspective, \eqref{HWMv} corresponds to a continuous version of the discrete completely integrable spin Calogero--Moser model \cite{lenzmann2024derivation, zhou2015solitons}. As an infinite-dimensional Hamiltonian system, \eqref{HWMv} possesses the energy functional \cite{lenzmann2018on,zhou2015solitons}
\begin{align*}
	E(\uu)=\frac{1}{2}\int_{\R}\uu\cdot|D|\uu dx=\frac{1}{4\pi}\int_{\R}\int_{\R}\frac{|\uu(x)-\uu(y)|^2}{|x-y|^2}dxdy,
\end{align*}
which is formally conserved along the flow of the equation. For any $\lambda>0$, the scaling $\uu(t,x)\mapsto\uu(\lambda t,\lambda x)$ leaves the equation \eqref{HWMv} and its energy invariant. Therefore, \eqref{HWMv} is energy-critical. Moreover, this equation is completely integrable and admits a Lax pair \cite{gerard2018a}. For a matrix-valued symbol ${A}$, let the Toeplitz operator
\begin{align*}
	T_{A}:L^2_+(\R; \C^2)\to L^2_+(\R; \C^2),\enspace f\mapsto T_{A}f:=\Pi_{+}({A} f),
\end{align*}
where $\Pi_+$ denotes the Cauchy--Szeg{\H{o}} projection onto $L^2_+$. Then for sufficiently regular solutions, the following Lax equation holds.
\begin{align*}
\frac{d}{dt}T_{\Ub(t)}&=[B_{\Ub(t)}^+,T_{\Ub(t)}],\enspace \Ub(t)=\uu(t)\cdot\bm\sigma,\enspace \bm\sigma=(\sigma_1,\sigma_2,\sigma_3),
\end{align*}
where $\sigma_j$ are the standard Pauli matrices, $[A,B]=AB-BA$, and
\begin{align*}
B_\Ub^+=-\frac{\ii}{2}\left(T_\Ub D+DT_\Ub\right)+\frac{\ii}{2}T_{|D|\Ub},\enspace D=-\ii\partial_x.
\end{align*}

Lenzmann and Schikorra \cite{lenzmann2018on} completely classified the finite-energy solitary wave solutions of \eqref{HWMv}. According to that work, \eqref{HWMv} admits a finite-energy solitary wave $\qv_v$ with a non-constant profile propagating at speed $v$ if and only if $v\in(-1,1)$. In addition, these solitary waves are all represented by finite Blaschke products, thus forming rational functions. If the degree of the corresponding Blaschke product is $m$, then the energy of $\qv_v$ is \cite{lenzmann2018on}
\begin{align*}
	E(\qv_v)=\left(1-v^2\right)\cdot m\pi.
\end{align*}
In particular, $m=1$ corresponds to the non-trivial solitary wave with the lowest energy, which is called the ground state solitary wave and has the form \cite{gerard2025global,lenzmann2018on}
\begin{align}\label{solitonv}
	\uu_v(t,x) = \qv_v(x-vt),
\end{align}
\begin{align*}
	\qv_{v}(x) =  \qv_\infty + \frac{\sv}{x- z} + \frac{\overline{\sv}}{x-\overline{z}},
\end{align*}
where $\qv_\infty\in \mathbb{S}^2$, $z\in \C_-$, and $\sv\in\C^3\setminus\{0\}$ satisfy
\begin{align*}
	\sv \cdot \sv = 0 \enspace \mbox{and} \enspace \sv \cdot \left ( \qv_\infty + \frac{\ov{\sv}}{z-\ov{z}} \right ) = 0.
\end{align*}
Here $\mathbf{a} \cdot \mathbf{b} = \sum_{k=1}^3 a_k b_k$ denotes the non-Hermitian dot product of $\mathbf{a}, \mathbf{b} \in \C^3$. Furthermore, Berntson, Klabbers, and Langmann \cite{berntson2020multi} obtained multi-solitons for \eqref{HWMv} by the spin Calogero-Moser system with certain constraints on initial conditions. Matsuno \cite{matsuno2022integrability} established analytic expressions for the multi-soliton solutions, and then showed that no phase shifts appear after the collision of solitons through asymptotic analysis.

For the equation \eqref{HWMv} on the real line, local well-posedness is known in the affine Sobolev space $H^s_\bullet$ for $s>3/2$ \cite{gerard2025global}. Recently, G\'{e}rard and Lenzmann \cite{gerard2026the} established global well-posedness in $H^{1/2}$ for the half-wave maps equation on the torus. In higher dimensions, small-data global existence for the half-wave maps equation on $\R^N$ with $N\geq3$ has been studied in \cite{kiesenhofer2021small,krieger2018small,liu2021global}. We now turn to rational maps, which occupy a special place in the dynamics of \eqref{HWMv} on the real line. For the set of rational functions
\begin{align*}
	\mathcal{R}at(\R;\mathbb{S}^2):=\left\lbrace \uu:\R\to\mathbb{S}^2\;|\;\uu(x)\; \text{is rational in}\; x\in\R \right\rbrace,
\end{align*}
it can be verified that for any $\uu\in\mathcal{R}at(\R;\mathbb{S}^2)$, there exists $\mathbf{p}\in\mathbb{S}^2$ such that
\begin{align*}
	\lim_{x\to\pm\infty}\uu(x)=\mathbf{p}.
\end{align*}
Furthermore, it can be verified that $\mathcal{R}at(\R;\mathbb{S}^2)\subset H^\infty_\bullet(\R;\mathbb{S}^2)$. It is worth noting that $\mathcal{R}at(\R;\mathbb{S}^2)$ is dense in $\dot{H}^{1/2}(\R;\mathbb{S}^2)$ \cite{gerard2025global}. According to the Lax pair and Kronecker's theorem, the flow of \eqref{HWMv} preserves rationality \cite{gerard2018a}. The global well-posedness of \eqref{HWMv} with rational initial data has been given by the following proposition:
\begin{proposition}{[G\'{e}rard and Lenzmann \cite{gerard2025global}]}\label{qjsdx}
	For any $\uu_0\in\mathcal{R}at(\R;\Ss^2)$, there exists a unique $\uu\in C(\R;H^\infty_\bullet(\R;\Ss^2))$ as a global solution of \eqref{HWMv} satisfying $\uu(0)=\uu_0$. Moreover, this solution has the form
	\begin{align*}
		\uu(t) = \uu_\infty + \mathbf{v}(t) \in \Ss^2 + C(\R; H^\infty(\R; \R^3)),
	\end{align*}
	where $\uu_\infty = \lim_{|x| \to \infty} \uu_0(x) \in \Ss^2$.
\end{proposition}
\noindent Note that the infinite conservation laws generated by the Lax pair for \eqref{HWMv} can only control one family of homogeneous Besov norms but is insufficient to control the Sobolev norms when $s>1/2$. G\'{e}rard and Lenzmann overcame this difficulty by establishing an explicit formula and combining it with algebraic constraints and the rationality of the initial data. On this basis, they further proved the soliton decomposition and non-turbulence of rational function solutions.
\begin{proposition}{[G\'{e}rard and Lenzmann \cite{gerard2025global}]}\label{ylczgzfj}
	Let
\begin{align*}
	\uu_0\in\mathcal{R}at(\R;\Ss^2),\enspace \Ub_0:=\uu_0\cdot\bm\sigma,\enspace \bm\sigma=(\sigma_1,\sigma_2,\sigma_3),
\end{align*}
and the Toeplitz operator
\begin{align*}
	T_{\Ub_0}:L^2_+(\R; \C^2)\to L^2_+(\R; \C^2),\enspace f\mapsto T_{\Ub_0}f:=\Pi_{+}(\Ub_0 f).
\end{align*}
Suppose that $T_{\Ub_0}$ has a simple discrete spectrum $\{ v_1, \ldots, v_N \}$. Then there exist ground state solitary wave profiles $\{ \qv_{v_1}, \ldots, \qv_{v_N}\}$ such that the global solution $\uu \in C(\R; H^\infty_\bullet(\R; \Ss^2))$ of \eqref{HWMv} with $\uu(0)=\uu_0$ satisfies
	\begin{align*}
		\lim_{t \to \pm \infty} \left\| \uu(t) - \left(\sum_{j=1}^N \qv_{v_j}(x-v_jt) - (N-1) \uu_\infty\right)  \right\|_{\dot{H}^s} = 0, \enspace \forall\;s>0,
	\end{align*}
	where $\uu_\infty = \lim_{|x| \to \infty} \uu_0(x) \in \Ss^2$. Moreover, the solution $\uu$ has the following a priori bound:
	\begin{align*}
		\sup_{t \in \R} \| \uu(t) \|_{\dot{H}^s} \leq C(\uu_0,s), \enspace \forall\; s>0.
	\end{align*}
\end{proposition}
\noindent This result shows that when the discrete spectrum is simple, the rational solutions resolve into finitely many ground state solitary waves with no remaining radiative term in the homogeneous Sobolev spaces. A Kronecker-type theorem shows that the initial data is rational if and only if the discrete spectrum of the Toeplitz operator is finite \cite{gerard2025global}. It is therefore natural to ask whether an infinite discrete spectrum can produce a soliton resolution into infinitely many solitary waves for non-rational solutions. This is precisely one of the objectives of the present work: to investigate the long-time behavior of such solutions and their soliton resolution. Another, more fundamental objective is to establish global well-posedness for \eqref{HWMv} on the real line for general initial data. In fact, the flow associated with rational initial data is known to be globally defined \cite{gerard2025global}. However, the long-time analysis in the general setting considered here requires global well-posedness for general data. To the best of our knowledge, this problem has remained open.

For these two objectives, we first review several approaches to the long-time asymptotic analysis of integrable equations. Among the earliest developments, Manakov, Ablowitz, and Newell \cite{ablowitz1973the,manakov1974nonlinear} pioneered the use of the inverse scattering transform (IST) for this purpose. Deift and Zhou \cite{deift1993a} later developed the nonlinear steepest descent method for Riemann--Hilbert problems and applied it to determine the long-time behavior of solutions for the modified Korteweg--de Vries equation. Since then, this method has been refined and extended to a variety of important nonlinear integrable equations \cite{charlier2023good,chen2021soliton,cuccagna2016asymptotic,de2020modified,de2009long,deift2011long}. In particular, the second author, Tian, and his collaborators investigated the soliton resolution conjecture for the Wadati--Konno--Ichikawa equation \cite{li2022soliton,li2022on}, the spin-$1$ Gross--Pitaevskii equation \cite{tian2026on}, and the modified Camassa--Holm equation \cite{yang2026the}.

The Hilbert transform in \eqref{HWMv} places it among nonlocal integrable equations. Related nonlocal structures also occur in the cubic {Szeg{\H{o}}} \cite{gerard2023unbounded,gerard2024szego}, Benjamin--Ono (BO) \cite{gassot2026proof,gerard2026lectures}, and Calogero--Moser derivative nonlinear Schr{\"o}dinger equations \cite{frank2025jost,gerard2024calogero}. Fokas and Ablowitz \cite{fokas1983inverse} formulated an IST for the BO equation based on a nonlocal Riemann--Hilbert problem. Ablowitz, Fokas, and Anderson \cite{ablowitz1983direct} also present a direct linearizing transform for this equation. A complementary route to the analysis of nonlocal integrable equations is using explicit formulas. G\'{e}rard and Grellier \cite{gerard2015an} derived such a formula for the cubic {Szeg{\H{o}}} equation, and this approach has since been extended to several other integrable equations involving the Hilbert transform \cite{gerard2023an,gerard2025global,gerard2024calogero}. Explicit formulas have also been used to investigate well-posedness and long-time behaviors for nonlocal equations \cite{badreddine2024global,gassot2026infinite,gassot2026proof,gerard2025global,gerard2026the,gerard2024szego,killip2025scaling}. For other results on the well-posedness of nonlocal integrable equations, see, for example, \cite{burq2008on,gerard2010cubic,gerard2023sharp,killip2024sharp,molinet2008global,saut1979sur}. For \eqref{HWMv} on the real line, the global well-posedness and soliton resolution with rational data in \cite{gerard2025global} exploit the finite-rank property of the Hankel operator associated with rational data, which reduces the explicit formula to a finite dimensional problem. This reduction is unavailable for general data.

For the global well-posedness with general data, our approach extends the work \cite{gerard2026the} for \eqref{HWMv} on the circle. We approximate the initial data by rational maps, for which the corresponding global classical solutions are given by the explicit formula. To pass to the limit in the explicit formula, we study the operator defined by its resolvent representation. A semigroup argument first shows that this operator is contractive, and a resolvent argument then establishes its convergence under perturbations of the Toeplitz operator. Passing to the limit in the rational intertwining relation then yields an intertwining identity for the limiting flow. Interestingly, the latter forces the resulting operator to be an isometry, which recovers the \(L^2\) conservation law. Combined with the trace identity for the Hankel operator, this also yields conservation of the \(\dot H^{1/2}\) energy and preservation of the discrete spectrum. By the strong convergence established, the explicit formula defines global solutions for general \(H^{1/2}_\bullet\)-data. Based on this, we further establish the continuous dependence of the global solution on the initial data, thereby proving the global well-posedness in $H^{1/2}_\bullet$.

We also prove that our global well-posedness result is sharp: the Cauchy problem for \eqref{HWMv} is ill-posed in $H^s_\bullet$ for any $s<1/2$. To prove sharpness for \eqref{HWMv}, we construct a sequence of rational sphere-valued initial data converging in every $H^s_\bullet$, $s<1/2$, to a sphere-valued step function with two jump discontinuities. At each fixed nonzero time, the corresponding classical solutions converge weakly in $L^2$ to an explicit profile. We compute this profile by reducing the limiting matrix resolvent to the scalar resolvent appearing in G\'erard's explicit formula for the zero-dispersion limit of the Benjamin--Ono equation \cite{gerard2024zero}. The limiting profile is not sphere-valued, which rules out a continuous sphere-valued extension of the classical rational flow to $H^s_\bullet$ for any $s<1/2$.

For the soliton resolution with general data, our strategy is inspired by \cite{gassot2026infinite,gassot2026proof}. By determining the subspace in which the eigenfunctions lie, we both establish a distorted Plancherel formula and apply the explicit formula to construct weak asymptotic profiles. With identities arising from the Toeplitz eigenvalue equation, we recover the geometric constraints and identify these profiles as the ground state solitary waves. To upgrade the soliton resolution from $\dot{H}^{1/2}$ to $L^\infty$ without conservation laws to control higher-order Sobolev norms, we develop a new trace norm estimate for the Hankel operator associated with the asymptotic error through finite-rank spectral truncations. Combining Peller's theorem with the Besov embedding, our argument yields the resolution into countably many ground state solitary waves for general \(H^{1/2}_\bullet\)-data under several assumptions: the discrete spectrum of the corresponding Toeplitz operator is required to be countable, simple, and satisfy a spectral summability condition. We call data meeting these requirements the admissible non-rational data.

To prove that the admissible non-rational data not only exist but also form a dense subset of the energy space $\dot{H}^{1/2}$, we recall that the rational data whose associated Toeplitz operators have finite simple discrete spectrum are dense in $\dot{H}^{1/2}$ \cite{gerard2025global}. More directly for our purpose, rational data for which the square of the corresponding Toeplitz operator has finite simple discrete spectrum also form a dense subset. It therefore suffices to find an admissible non-rational data in every neighborhood of each rational data in this class. More precisely, near every such rational data, we first construct another rational data in the same class for which the squared Toeplitz operator has one additional discrete spectrum point. The perturbation is controlled both in Sobolev norms and in the trace norm of the corresponding Hankel operators. Iterating this step produces a sequence of rational data whose limit is non-rational, satisfies the required spectral conditions, and remains arbitrarily close to the first data.

Before presenting our main results, we collect here the conventions used throughout the paper. We write
\begin{align*}
	\R_+=(0,+\infty),\enspace \R_-=(-\infty,0),\enspace
	\C_\pm=\left\{z\in\C:\ \pm\operatorname{Im}z>0\right\},
\end{align*}
and set $\N=\{0,1,2,\ldots\}$ and $\N^*=\{1,2,\ldots\}$. The symbol $\bm\sigma$ denotes $(\sigma_1,\sigma_2,\sigma_3)$, where $\sigma_1$, $\sigma_2$, and $\sigma_3$ are the standard Pauli matrices. The symbol $\mathcal V$ denotes either $\C^d$ or $M_d(\C)=\C^{d\times d}$, according to the context. These spaces are equipped with the inner products
\begin{align*}
	\langle a,b\rangle_{\C^d}=b^*a,\enspace\langle A,B\rangle_{M_d(\C)}=\Tr(AB^*).
\end{align*}
Their induced norms are denoted by
$|\cdot|_{\mathcal V}$. In addition, the function space $L^2(\R;\mathcal V)$ is equipped with the inner product
\begin{align*}
	\langle f,g\rangle_{L^2}:=\int_{\R}\langle f(x),g(x)\rangle_{\mathcal V}\,\mathrm{d}x.
\end{align*}
For matrices and operators, $^*$ denotes the adjoint and $[A,B]=AB-BA$ the commutator.

The Fourier transform and inverse Fourier transform are normalized by
\begin{align*}
	\widehat f(\xi)=\mathcal{F}[f](\xi)
	:=\int_{\R}\eu^{-\ii x\xi}f(x)\,\mathrm{d}x,\enspace
	f(x)=\frac{1}{2\pi}\int_{\R}\eu^{\ii x\xi}\widehat f(\xi)\,\mathrm{d}\xi.
\end{align*}
The transform acts componentwise on vector-valued and matrix-valued functions. We use the Hardy space
\begin{align*}
	L^2_+(\R;\mathcal V)
	:=\big\{f\in L^2(\R;\mathcal V):\;
	\operatorname{supp}\widehat f\subset [0,+\infty)\big\}
\end{align*}
and the Cauchy--Szeg{\H{o}} projections
\begin{align*}
	\Pi_+:L^2(\R; \mathcal{V})\to L^2_+(\R; \mathcal{V}),\enspace f(x)\mapsto (\Pi_+f)(x):=\frac{1}{2\pi}\int_0^\infty
	\eu^{\ii x\xi}\widehat f(\xi)\,\mathrm{d}\xi.
\end{align*}
So that $\widehat{(\Pi_+ f)}(\xi)=\mathbbm{1}_{\xi\geq0}\widehat f(\xi)$. Furthermore, the space $L^2_-(\R; \mathcal{V}):=\left(L^2_+(\R; \mathcal{V})\right)^\perp$ and the operator $\Pi_-:=\text{Id}-\Pi_+$.

For $s\geq0$, the Sobolev norms are
\begin{align*}
	\|f\|_{\dot H^s}
	&:=\||D|^sf\|_{L^2}=\left[\frac{1}{2\pi}\int_{\R}|\xi|^{2s}
	\left|\widehat f(\xi)\right|_{\mathcal V}^2\,\mathrm{d}\xi\right]^{1/2},\\
	\|f\|_{H^s}
	&:=\|\langle D\rangle^sf\|_{L^2}=\left[\frac{1}{2\pi}\int_{\R}\langle\xi\rangle^{2s}
	\left|\widehat f(\xi)\right|_{\mathcal V}^2\,\mathrm{d}\xi\right]^{1/2},\enspace \langle\xi\rangle=\left(1+|\xi|^2\right)^{1/2}.
\end{align*}
The affine Sobolev space is
\begin{align*}
	H^s_\bullet(\R; \mathcal{V}) := \mathcal{V} \oplus H^s(\R; \mathcal{V}),\enspace H^\infty_\bullet(\R;\mathcal{V}) := \cap_{s > 0} H^s_\bullet(\R;\mathcal{V}).
\end{align*}
Furthermore, one set
\begin{align*}
	H^s_\bullet(\R; \mathcal M):= \left \{ f\in H^s_\bullet(\R; \mathcal V) \mid \mbox{$f(x) \in \mathcal M$ for a.\,e.~$x\in \R$} \right \},
\end{align*}
where $\mathcal M$ denotes either $\mathcal M=\mathbb S^2\subset \mathbb R^3$ or later the complex Grassmannians. Moreover, we denote by $C_0(\R;\mathcal V)$ the space of continuous functions vanishing at infinity.

Finally, let \(\mathcal X\) and \(\mathcal Y\) be separable Hilbert spaces, let \(A:\mathcal X\to\mathcal Y\) be a compact operator, and let \(s_1(A)\ge s_2(A)\ge\cdots>0\) be the singular values of \(A\). The Schatten class norm of \(A\) is defined by
\begin{align*}
	\|A\|_{\mathfrak S_1}=\sum_{j=1}^\infty s_j(A),
	\enspace
	\|A\|_{\mathfrak S_2}
	=\left(\sum_{j=1}^\infty s_j(A)^2\right)^{1/2},
\end{align*}
where \(\mathfrak S_1\) denotes the trace norm and \(\mathfrak S_2\) the Hilbert--Schmidt norm.

We now present our main results, which are divided into two parts. In the first part, we first establish the global well-posedness result in $H^{1/2}_\bullet$. In addition, using jump-function initial data as a counterexample, we prove the ill-posedness in $H^{s}_\bullet$ for every $s<1/2$, which implies that the global well-posedness in $H^{1/2}_\bullet$ is sharp.

On the sharp well-posedness, we state some preliminaries to introduce the results in this direction. Since Lemma \ref{rational fixed background} shows that rational maps are dense in $H^{1/2}_\bullet(\R;\Ss^2)$, we define the global solutions by approximation with the global classical rational solutions.

\begin{definition}\label{def global solution}
Let $\uu_0=\uu_\infty+\mathbf v_0\in H^{1/2}_\bullet(\R;\Ss^2)$, where $\uu_\infty\in\Ss^2$ and $\mathbf v_0\in H^{1/2}(\R;\R^3)$. A curve
\begin{align*}
 \uu=\uu_\infty+\mathbf v\in C(\R;H^{1/2}_\bullet(\R;\Ss^2)),\enspace \uu(0)=\uu_0,
\end{align*}
is called a \textit{global-in-time solution} of \eqref{HWMv} in $H^{1/2}_\bullet$ with initial data $\uu_0$ if, for every sequence
\begin{align*}
 \uu_{0,n}=\uu_\infty+\mathbf v_{0,n}\in\mathcal{R}at(\R;\Ss^2),\enspace
 \|\uu_{0,n}-\uu_0\|_{H^{1/2}}\to0,
\end{align*}
the corresponding global classical rational solutions $\uu_n=\uu_\infty+\mathbf v_n$ with $\uu_n(0)=\uu_{0,n}$ satisfy
\begin{align*}
 \sup_{|t|\leq T}\|\uu_n(t)-\uu(t)\|_{H^{1/2}}\to0,
 \enspace\text{for every }T>0.
\end{align*}
We denote the global-in-time solution by $\Phi_\cdot(\uu_0)$, with $\Phi_t(\uu_0):=\uu(t)$.
\end{definition}
Definition \ref{def global solution} implies that a global-in-time solution, if it exists, is unique. The corresponding global well-posedness is defined as follows.
\begin{definition}\label{def global wellposedness}
Equation \eqref{HWMv} is said to be \textit{globally well-posed} in $H^{1/2}_\bullet(\R;\Ss^2)$ if the following hold:
\begin{itemize}
\item[(i)] For every $\uu_0=\uu_\infty+\mathbf v_0\in H^{1/2}_\bullet(\R;\Ss^2)$, there exists a global-in-time solution with initial data $\uu_0$ in the sense of Definition \ref{def global solution}.
\item[(ii)] The map
\begin{align*}
 \uu_0\mapsto\Phi_\cdot(\uu_0):
 H^{1/2}_\bullet(\R;\Ss^2)\to C(\R;H^{1/2}_\bullet(\R;\Ss^2))
\end{align*}
is continuous. More precisely, for every sequence
\begin{align*}
	\uu_{0,n}=\uu_\infty+\mathbf v_{0,n}\in H^{1/2}_\bullet(\R;\Ss^2),\enspace
	\|\uu_{0,n}-\uu_0\|_{H^{1/2}}\to0,
\end{align*}
the corresponding global-in-time solutions satisfy
\begin{align*}
	\sup_{|t|\leq T}\|\Phi_t(\uu_{0,n})-\Phi_t(\uu_0)\|_{H^{1/2}_\bullet}\to0,\enspace\text{for every }T>0.
\end{align*}
\end{itemize}
\end{definition}
We present the following result concerning the global well-posedness for \eqref{HWMv}.
\begin{theorem}[\textbf{Global well-posedness in $H^{1/2}_\bullet$}]\label{global energy}
Equation \eqref{HWMv} is globally well-posed in $H^{1/2}_\bullet(\R;\Ss^2)$ in the sense of Definition \ref{def global wellposedness}. For every initial data
\begin{align*}
	\uu_0=\uu_\infty+\mathbf v_0\in H^{1/2}_\bullet(\R;\Ss^2),
\end{align*}
the corresponding solution
\begin{align*}
 \uu(t)=\Phi_t(\uu_0)=\uu_\infty+\mathbf v(t),\enspace
 \mathbf v\in C(\R;H^{1/2}(\R;\R^3)),
\end{align*}
solves \eqref{HWMv} in the distributional sense, and the following properties hold.
\begin{itemize}
	\item[(i)] The following functionals are conserved.
	\begin{align*}
		\|\mathbf v(t)\|_{L^2}=\|\mathbf v_0\|_{L^2},\enspace \|\mathbf v(t)\|_{\dot H^{1/2}}=\|\mathbf v_0\|_{\dot H^{1/2}},\enspace t\in\R.
	\end{align*}
	\item[(ii)] Let $\Ub(t)=\Phi_t(\uu_0)\cdot\bm\sigma$. The discrete spectrum of the Toeplitz operator
	\begin{align*}
		T_{\Ub(t)}:L^2_+(\R; \C^2)\to L^2_+(\R; \C^2),\enspace f\mapsto T_{\Ub(t)}f:=\Pi_{+}(\Ub f)
	\end{align*}
	is preserved.
	\item[(iii)] The flow maps satisfy
	\begin{align*}
		\Phi_0=\id,\enspace \Phi_{t+\tau}=\Phi_t\circ\Phi_\tau,\enspace t,\tau\in\R,
	\end{align*}
	and agree with the global classical rational flow on $\mathcal{R}at(\R;\Ss^2)$.
\end{itemize}
\end{theorem}
\begin{remark}
For a more general statement, see Proposition \ref{global solution matrix} in Section \ref{p of 0}, where we establish this global well-posedness for the generalized half-wave maps equation. For more details about this equation, see Section \ref{Preliminaries}.
\end{remark}
\begin{remark}
Theorem \ref{global energy} is closely related to the global well-posedness in $\dot H^{1/2}$ announced without proof by G\'{e}rard and Lenzmann in \cite{gerard2026the}. Both results concern the critical energy regularity. The extra assumption here is $\uu_0-\uu_\infty\in L^2$. This $L^2$ regularity of the perturbation is preserved for all time. We briefly describe how the argument in Section \ref{p of 0} can be adapted to remove this assumption.

Equip $\dot H^{1/2}(\R;\Gr_k(\C^d))$ with the distance
\begin{align*}
 d(\Ub,\widetilde\Ub)^2
 =\|\Ub-\widetilde\Ub\|_{\dot H^{1/2}}^2
 +\int_{\R}\frac{|\Ub(x)-\widetilde\Ub(x)|_F^2}{1+x^2}\,dx.
\end{align*}
The second term is finite for all sphere-valued maps. Rational density in this distance follows from the Cayley transform and approximation in $H^{1/2}$ on the circle; see \cite[Appendix A]{gerard2025global}.

Energy conservation and the target constraint give local $L^2$ compactness of the rational solutions. For each local limit, the intertwining relation in \eqref{rational explicit} passes to the limit. The conclusion of Lemma \ref{OtTisometry} can be strengthened to unitarity. Indeed, the semigroup representation of $\mathcal O_{tT}$ shows that its closed range reduces $\eu^{-\ii\eta X^*}$ and that $I_+(\mathcal O_{tT}F)=I_+(F)$ for $F\in\dom(X^*)$. The boundary identity for $X^*$ then forces the orthogonal complement of this range to be zero. Thus the limiting intertwining relation gives
\begin{align*}
 T_{\Ub(t)}=\mathcal O_{tT_{\Ub_0}}T_{\Ub_0}\mathcal O_{tT_{\Ub_0}}^*.
\end{align*}
A bounded symbol is uniquely determined by its Toeplitz operator, so this identity identifies every local limit without applying $\mathcal O_{tT_{\Ub_0}}$ to an $L^2$ perturbation of the initial data. The identities $K_\Ub=\id-T_\Ub^2$ and $\|\Ub\|_{\dot H^{1/2}}^2=4\pi\Tr K_\Ub$ then give energy conservation and strong convergence in the distance above. The remaining continuity and flow arguments follow as in Section \ref{p of 0}.
\end{remark}

We next show that the global well-posedness result we established is sharp.
\begin{theorem}[\textbf{Ill-posedness below $H^{1/2}_\bullet$}]\label{thm illposedness}
Let the sequence
\begin{align*}%\label{sharp rational data}
 \mathbf u_{0,n}(x)=\left(\frac{2x^{2n}}{1+x^{4n}},\,0,\,\frac{x^{4n}-1}{1+x^{4n}}\right)\in\mathcal Rat(\R;\Ss^2),\enspace n\ge1.
\end{align*}
Then
\begin{align}\label{sharp initial convergence}
 \lim_{n\to\infty}\|\mathbf u_{0,n}-f(\cdot)\mathbf e_3\|_{H^s}=0\enspace\text{for every }s<1/2,\enspace f=1-2\mathbbm1_I,\enspace I=(-1,1),
\end{align}
where $\mathbf e_3=(0,0,1)^\top$. Let $\mathbf u_n(t)=\Phi_t(\mathbf u_{0,n})=\mathbf e_3+\mathbf v_n(t)$, then for every $t\ne0$,
\begin{align}\label{sharp weak limit}
 \mathbf v_n(t)\rightharpoonup (h_t-1)\mathbf e_3\enspace\text{in }L^2(\mathbb R;\mathbb R^3),\enspace h_t(x)=\frac{1}{2|t|}\int_{x-|t|}^{x+|t|}f(y)\,dy.
\end{align}
The function $h_t\mathbf e_3$ is not sphere-valued. Furthermore,
\begin{align}\label{sharp defect}
 \lim_{n\to\infty}\|\mathbf u_n(t)-h_t\mathbf e_3\|_{L^2}^2=\frac{8|t|}{3}\enspace \text{for }\; 0<|t|<1.
\end{align}
\end{theorem}
\begin{remark}
The same example gives ill-posedness for the generalized half-wave maps equation. For more details about this equation, see Section \ref{Preliminaries}. In fact, for every $1\leq k\leq d-1$, using the Pauli matrices from \eqref{relation}, set
\begin{align*}
 \Ub_n(t,x)=\operatorname{diag}\bigl(\mathbf u_n(t,x)\cdot\bm\sigma,\mathbbm1_{d-k-1},-\mathbbm1_{k-1}\bigr).
\end{align*}
Each $\Ub_n$ takes values in $\Gr_k(\C^d)$ and solves \eqref{HWMm}. Their initial data converge to
\begin{align*}
 \Ub_*(x)=\operatorname{diag}\bigl(f(x)\sigma_3,\mathbbm1_{d-k-1},-\mathbbm1_{k-1}\bigr).
\end{align*}
in $H^s_\bullet$ for every $s<1/2$. At every $t\ne0$, their distributional limit is
\begin{align*}
 \widetilde\Ub(t,x)=\operatorname{diag}\bigl(h_t(x)\sigma_3,\mathbbm1_{d-k-1},-\mathbbm1_{k-1}\bigr).
\end{align*}
Hence $\widetilde\Ub^2\ne\mathbbm1_d$ on a set of positive measure, which implies that this limit leaves $\Gr_k(\C^d)$.
\end{remark}

In the second part, under the necessary spectral assumptions, we study the long-time behavior of solutions with general initial data in $H^{1/2}_\bullet$. Interestingly, we find that non-rational solutions always asymptotically approach an infinite superposition of independent soliton solutions with different speeds. We also prove the density in the energy space of initial data satisfying the spectral assumptions.

To put this in context, recall that rational solutions of \eqref{HWMv} asymptotically approach the sum of finitely many solitons. In contrast, we show that for initial data in $H^{1/2}_\bullet$ other than rational functions, the corresponding solutions asymptotically approach an infinite soliton series. Together with the density result, this implies the establishment of the soliton resolution result for general initial data in $H^{1/2}_\bullet$.
\begin{theorem}[\textbf{Soliton resolution in the $\dot{H}^{1/2}$ seminorm}]\label{infinitysolitons}
	Let
\begin{align*}
\uu_0=\uu_\infty+\mathbf v_0\in H^{1/2}_\bullet(\R;\Ss^2),\enspace \uu_0\notin\mathcal{R}at(\R;\Ss^2),\enspace \Ub_0=\uu_0\cdot\bm\sigma.
\end{align*}
Suppose that the discrete spectrum $(v_n)_{1\leq n\leq\infty}$ of the Toeplitz operator
\begin{align*}
	T_{\Ub_0}:L^2_+(\R; \C^2)\to L^2_+(\R; \C^2),\enspace f\mapsto T_{\Ub_0}f:=\Pi_{+}(\Ub_0 f)
\end{align*}
is simple and satisfies
	\[
		\sum_{n=1}^\infty \sqrt{1-v_n^2}<\infty.
	\]
	Then there exist ground state solitary wave profiles $( \qv_{v_n})_{1\leq n\leq\infty}$ such that
	\begin{align*}
		\lim_{m\to\infty}\limsup_{t \to \pm \infty} \left\| \mathbf{v}(t) - \sum_{n=1}^m \left(\qv_{v_n}(x-v_nt)-\uu_\infty\right)  \right\|_{H^\frac{1}{2}} = 0
	\end{align*}
	and
	\begin{align*}
		\lim_{t \to \pm \infty} \left\| \mathbf{v}(t) - \sum_{n=1}^\infty \left(\qv_{v_n}(x-v_nt)-\uu_\infty\right)  \right\|_{\dot{H}^\frac{1}{2}} = 0.
	\end{align*}
\end{theorem}
\begin{remark}
Theorem \ref{infinitysolitons} shows the soliton resolution for initial data in the critical space other than rational data. Together with the soliton resolution for rational data established in \cite{gerard2025global}, this establishes the soliton resolution result for general initial data in the critical space.
\end{remark}
\begin{remark}
For a more general statement, see Propositions \ref{infinitysolitons general} in Section \ref{p of 1}, where we establish this soliton resolution for the generalized half-wave maps equation. For more details about this equation, see Section \ref{Preliminaries}.
\end{remark}
\begin{remark}
	The spectral summability condition ensure the convergence of the soliton series in the energy space and also ensure its absolute uniform convergence. It also implies $\mathbf v_0\in C_0$ by Lemma \ref{spectral Linfty}. This supplies the spatial decay needed in the spectral analysis without an additional Sobolev assumption above $H^{1/2}$.
\end{remark}
The following result shows that the soliton resolution also holds in the $L^\infty$ norm.
\begin{theorem}[\textbf{Soliton resolution in the $L^{\infty}$ norm}]\label{infinitysolitonsLinfty}
	Under the assumptions and with the notation of Theorem \ref{infinitysolitons}, the series
	\begin{align*}
		\sum_{n=1}^\infty \left(\qv_{v_n}(x-v_nt)-\uu_\infty\right)
	\end{align*}
	converges absolutely and uniformly in $x\in\R$, with the convergence uniform in $t\in\R$. Furthermore, we have
	\begin{align*}
		\lim_{t \to \pm \infty} \left\| \mathbf{v}(t) - \sum_{n=1}^\infty \left(\qv_{v_n}(x-v_nt)-\uu_\infty\right)  \right\|_{L^\infty} = 0.
	\end{align*}
\end{theorem}
\begin{remark}
For a more general statement, see Propositions \ref{infinitysolitonsLinfty general} in Section \ref{p of 2}, where we establish this soliton resolution for the generalized half-wave maps equation. For more details about this equation, see Section \ref{Preliminaries}.
\end{remark}
\begin{remark}
In view of the requirements of soliton resolution, the initial data considered in Theorems \ref{infinitysolitons} and \ref{infinitysolitonsLinfty} are equipped with the spectral condition. We will show that such admissible non-rational data not only exist but also form a dense subset of the energy space $\dot{H}^{1/2}(\R;\Ss^2)$.
\end{remark}
Let $\uu\in {H}_\bullet^{1/2}(\R;\Ss^2)$, $\Ub=\uu\cdot\bm{\sigma}$, and the associated operator
\begin{align*}
K_{\Ub}:L^2_+(\R; \C^2)\to L^2_+(\R; \C^2),\enspace f\mapsto K_{\Ub}f:=H^*_{\Ub}H_{\Ub}f.
\end{align*}
The following result provides the existence of the admissible non-rational data.
\begin{theorem}[\textbf{Examples of the admissible non-rational data}]\label{finite to infinite}
Let
\begin{align*}
\uu_{*,0}=\uu_\infty+\mathbf v_{*,0}\in\mathcal{R}at(\R;\Ss^2),\enspace \Ub_{*,0}=\uu_{*,0}\cdot\bm\sigma,
\end{align*}
and suppose that all positive eigenvalues of \(K_{\Ub_{*,0}}\) are simple. Then for every \(m\in\N^*\) and every
\(\varepsilon>0\), there exists
\begin{align*}
\uu_0=\uu_\infty+\mathbf v_0\in H^\infty_\bullet(\R;\Ss^2),\enspace \Ub_0=\uu_0\cdot\bm\sigma,
\end{align*}
such that
\begin{align*}
\left\|\mathbf v_0-\mathbf v_{*,0}\right\|_{H^m}<\varepsilon,
\end{align*}
and \(K_{\Ub_0}\) has infinite simple positive spectrum $(\lambda_n)_{1\leq n\leq\infty}$, which satisfies that
\begin{align*}
\sum_{n=1}^\infty\sqrt{\lambda_n}<\infty.
\end{align*}
Furthermore, the global solution $\uu$ for \eqref{HWMv} with $\uu(0)=\uu_0$ is well defined in $C(\R; H^\infty_\bullet(\R;\Ss^2))$.
\end{theorem}
\begin{remark}
By means of the identity \(T_{\Ub}^2=\id-K_{\Ub}\) \cite{gerard2025global}, together with the properties of \(K_{\Ub_0}\) in Theorem \ref{finite to infinite}, it follows that the discrete spectrum \((v_n)_{1\le n\le\infty}\) of \(T_{\Ub_0}\) is countable, simple, and satisfies
\[
	\sum_{n=1}^\infty \sqrt{1-v_n^2}<\infty.
\]
Therefore, we first obtain the density of the admissible non-rational data in a class of rational data.
\end{remark}
Since the set
\begin{align*}
\mathcal Rat_{\mathrm{ss}}(\R;\Ss^2):=\left\{\uu\in\mathcal Rat(\R;\Ss^2):K_\Ub\text{ has simple positive spectrum}\right\}
\end{align*}
is dense in \(\dot H^{1/2}(\R;\Ss^2)\) \cite{gerard2025global}, we obtain the following corollary of Theorem \ref{finite to infinite}.
\begin{corollary}
The admissible non-rational data are dense in \(\dot{H}^{1/2}(\mathbb{R};\mathbb{S}^2)\).
\end{corollary}

The structure of the paper is as follows. In Section \ref{Preliminaries}, we recall some known preliminaries concerning \eqref{HWMv}, including the generalized half-wave maps equation, the Lax operator and its properties, well-posedness, and explicit formulas. In Section \ref{p of 0}, we establish the isometry property of the transform in the explicit formula, prove continuous dependence, and complete the proof of Theorem \ref{global energy}. In Section \ref{sec illposedness}, we prove the matching ill-posedness theorem for $s<1/2$. In Section \ref{section eigenfunctions}, we establish some foundational results, including a proof that the eigenfunctions of the Toeplitz operator belong to the domain of $X^{*}$, as well as certain identities derived from this fact, such as the distorted Plancherel formula. In Section \ref{section weak}, we use explicit formulas to construct weak convergence limits of solutions with general initial data and prove that the limiting profiles constitute ground state solitary waves. In Section \ref{p of 1}, we estimate the cross terms and remainders for finite soliton sums to prove Theorem \ref{infinitysolitons}. In Section \ref{p of 2}, we study Hankel singular values and derive a trace norm estimate by finite-rank truncation, proving Theorem \ref{infinitysolitonsLinfty}. Finally, in Section \ref{p of 3}, we construct a perturbation that increases the rank by one while preserving spectral simplicity and establish a stability result on bounded time intervals to prove Theorem \ref{finite to infinite}.

\section{Preliminaries}\label{Preliminaries}
We recall the matrix formulation, Lax operator, and explicit formula for \eqref{HWMv}.

Let $d\geq2$. We consider the generalized half-wave maps equation \cite{berntson2022spin,gerard2018a,gerard2025global}
\begin{align}\label{HWMm}\tag{HWM$_d$}
	\partial_t\Ub=-\frac{i}{2}[\Ub,|D|\Ub],\enspace \Ub: [0,T)\times\R \to \bigcup_{k=0}^{d}\Gr_k(\mathbb{\C}^d),
\end{align}
where
$$
\Gr_k(\C^d) \cong \left \{ U \in M_d(\C) \mid \mbox{$U=U^*, \, U^2 = \mathbbm{1}_{d}$ and $\Tr(U) = d- 2k$} \right \}.
$$
G\'{e}rard and Lenzmann \cite{gerard2025global} systematically investigated \eqref{HWMm}, which is a more convenient object to study than \eqref{HWMv}. They first established global well-posedness, soliton resolution, and non-turbulence for \eqref{HWMm} with rational initial data, and then obtained the corresponding results for \eqref{HWMv}, which is a special case of \eqref{HWMm}. In fact, since the property $\Ub(x)\in \Gr_k(\C^d)$ for a.e. $x\in\R$ is formally preserved along the flow of \eqref{HWMm}, one can restrict the target of \eqref{HWMm} to $\Gr_k(\C^d)$ for any given $k=0,1,\ldots,d$. In particular, for the target $\Gr_1(\C^2)$, \eqref{HWMm} is equivalent to the standard \eqref{HWMv} through the following relation:
\begin{align}\label{relation}
	\Ub=\uu\cdot\bm\sigma=\sum_{k=1}^3u_k\sigma_k=\begin{pmatrix}
		u_3 & u_1-iu_2 \\
		u_1+iu_2 & -u_3
	\end{pmatrix}\in \Gr_1(\C^2),\enspace \uu=\left(u_1,u_2,u_3\right)^\top\in \mathbb{S}^2,
\end{align}
where $\sigma_1$, $\sigma_2$, and $\sigma_3$ are the standard Pauli matrices.

The Hamiltonian energy functional of \eqref{HWMm} is \cite{gerard2025global}
$$
E(\Ub)=\frac{1}{2}\|\Ub\|_{\dot{H}^\frac{1}{2}}^2=\frac{1}{4\pi}\int_{\R}\int_{\R}\frac{|\Ub(x)-\Ub(y)|_{F}^2}{|x-y|^2}dxdy.
$$
For $d=2$, $k=1$, \eqref{relation} gives $\|\Ub\|_{\dot H^s}=\sqrt2\|\uu\|_{\dot H^s}$ for every $s>0$.

For the speed $v\in(-1,1)$, \eqref{HWMm} admits traveling solitary waves of the form \cite{gerard2025global,lenzmann2018on}
\begin{align}\label{solitonm}
	\Ub_v(t,x) = \Qv_v(x-vt),
\end{align} $\Qv_{v}(x) = \Ub_\infty + \frac{A}{x- y + \ii \delta} + \frac{A^*}{x-y- \ii \delta}$, where $\Ub_\infty\in \Gr_k(\C^d)$, $y\in\R$, $\delta>0$, and $A\in M_d(\C)$ satisfy
$$
A\neq0,\enspace \rank (A)=1,\enspace A^2=0,\enspace A ( \Ub_\infty + \frac{\ii A^*}{2\delta} ) + ( \Ub_\infty + \frac{\ii A^*}{2\delta} ) A = 0.
$$
For $d=2$, $k=1$, \eqref{relation} identifies \eqref{solitonm} with the ground state solitary wave \eqref{solitonv}.

Let $s\geq1/2$,
\begin{align}\label{stronginitial 1}
	\Ub(x)= \Ub_\infty + \mathbf{V}(x)\in M_d(\C) \oplus H^s(\R; M_d(\C)),
\end{align}
\begin{align}\label{stronginitial 2}
	\Ub(x)\in \Gr_k(\C^d),\enspace \text{for a.e. }x\in\R.
\end{align}
Then the Toeplitz operator
$$
T_{\Ub}:L^2_+(\R; \mathcal{V})\to L^2_+(\R; \mathcal{V}),\enspace f\mapsto T_{\Ub}f:=\Pi_{+}(\Ub f)
$$
is self-adjoint with $\|T_{\Ub}\|\leq1$. Its spectrum is preserved along the classical flow of \eqref{HWMm} \cite{gerard2018a,gerard2025global}; \eqref{relation} gives the Lax operator for \eqref{HWMv}.

Let the Hankel operator $H_{\Ub}$ be
$$
H_{\Ub}:L^2_+(\R; \mathcal{V})\to L^2_-(\R; \mathcal{V}),\enspace f\mapsto H_{\Ub}f:=\Pi_{-}(\Ub f).
$$
Its adjoint is
$$
H^*_{\Ub}:L^2_-(\R; \mathcal{V})\to L^2_+(\R; \mathcal{V}),\enspace f\mapsto H^*_{\Ub}f:=\Pi_{+}(\Ub f).
$$
Then the operator
\begin{align}\label{KandH}
	K_{\Ub}:L^2_+(\R; \mathcal{V})\to L^2_+(\R; \mathcal{V}),\enspace f\mapsto K_{\Ub}f:=H^*_{\Ub}H_{\Ub}f
\end{align}
is self-adjoint, nonnegative, bounded, trace-class, and satisfies the identity \cite{gerard2025global}
\begin{align}\label{TandK}
	T_{\Ub}^2=\id-K_{\Ub}\enspace \text{on}\;\; L^2_+(\R; \mathcal{V}).
\end{align}
Define the closed subspaces
\begin{align}\label{subspace}
	\Hfr_0 := \ker(K_{\Ub}) \enspace \mbox{and} \enspace \Hfr_1 := \ov{\ran(K_{\Ub})},
\end{align}
so that \cite{gerard2025global} $L^2_+(\R; \mathcal{V})=\Hfr_0 \oplus \Hfr_1,\enspace \Hfr_0 \perp \Hfr_1$. In addition, both $\Hfr_0$ and $\Hfr_1$ are invariant subspaces of the operator $T_{\Ub}$. Moreover, $\Hfr_0=\ker(H_{\Ub})$ and $\Hfr_1=\ov{\ran(H^*_{\Ub})}$.

The spectrum of $T_{\Ub}$ can be decomposed as \cite{gerard2025global} $\sigma(T_{\Ub})=\sigma_\text{e}(T_{\Ub})\cup\sigma_\text{d}(T_{\Ub})$, where the essential spectrum ${\sigma }_{\mathrm{e}}\left( {T}_{\Ub}\right) \subset \{ \pm 1\}$, and the discrete spectrum
$$
{\sigma }_{\mathrm{d}}\left( {T}_{\Ub}\right)  = {\sigma }_{\mathrm{d}}\left( {\left. {T}_{\Ub}\right| }_{{\mathfrak{H}}_{1}}\right)  = \left\{  {{\varepsilon }_{n}\sqrt{1 - {\lambda }_{n}} \mid  n = 1,\ldots ,\omega}\right\}\subset(-1,1),\enspace \omega=\operatorname{rank}\left( {K}_{\Ub}\right)\in\mathbb{N}\cup\{\infty\},
$$
where for any $n$, one has ${\varepsilon }_{n}\in\{\pm 1\}$. In addition, $\left\{ {\lambda }_{n}\right\}_{1\leq n\leq\omega}  \subset  (0,1\rbrack$ is the sequence of discrete spectrum points of the compact self-adjoint operator $K_{\Ub}$.

The operator $X^*$ is defined as
$$
X^*:\dom(X^*)\to L^2_+(\R; \mathcal{V}),\enspace f\mapsto X^*f,
$$
where
$$
\dom(X^*) = \left\{f\in L^2_+(\R; \mathcal{V}) \;\big|\; \widehat{f} \in H^1(\R_+; \mathcal{V}) \right\},\enspace \widehat{X^* f}(\xi) := \ii \frac{\mathrm{d}}{\mathrm{d} \xi} \widehat{f}(\xi)\mathbbm{1}_{\xi\geq0}.
$$
The operator $I_+$ is defined as $I_+:\dom(X^*)\to \mathcal{V},\enspace f\mapsto I_+(f):= \widehat{f}(0^+)$. The local well-posedness and explicit formula of \eqref{HWMm} are given by the following proposition.
\begin{proposition}{[G\'{e}rard and Lenzmann \cite{gerard2025global}]}\label{local explicit}
	Let $\Ub_0=\Ub_\infty+\mathbf{V}_0$ satisfy the conditions \eqref{stronginitial 1} and \eqref{stronginitial 2} with $s>3/2$. For any $R>0$, there exists $T=T(R)>0$, such that for every such $\Ub_0=\Ub_\infty+\mathbf{V}_0$ satisfying $\|\mathbf{V}_0\|_{H^s}<R$, there exists a unique solution $\Ub\in C([0,T];H^s_\bullet(\R;\Gr_k(\C^d)))$ to \eqref{HWMm} with $\Ub(0)=\Ub_0$. Moreover, this solution has the form
	$\Ub(t)= \Ub_\infty + \mathbf{V}(t)\in M_d(\C) \oplus C([0,T];H^s(\R; M_d(\C)))$.
	In addition,
	\begin{align}\label{explicit formula}
		\Pi_+ \mathbf{V}(t, z) = \frac{1}{2\pi \ii} I_+ \big[ (X^* + t T_{\Ub_0} - z \id)^{-1} \Pi_+ \mathbf{V}_0 \big],\enspace \forall t\in [0,T],\enspace z\in\C_+.
	\end{align}
\end{proposition}

\section{Proof of Theorem \ref{global energy}: Global well-posedness in $H^{1/2}_\bullet$}\label{p of 0}
We first approximate the initial data by rational maps.
\begin{lemma}\label{rational fixed background}
Let $s\geq1/2$ and $\Ub_0=\Ub_\infty+\mathbf V_0\in H^s_\bullet(\R;\Gr_k(\C^d))$. Then there exists a sequence $\Ub_{0,j}=\Ub_\infty+\mathbf V_{0,j}\in\mathcal{R}at(\R;\Gr_k(\C^d))$ such that $\|\mathbf V_{0,j}-\mathbf V_0\|_{H^s}\to 0$. \end{lemma}
\begin{proof}
The cases $k=0,d$ are constant maps. Suppose first that $s>1/2$. Approximate $\mathbf V_0$ in $H^s$ by smooth and compactly supported Hermitian matrix-valued functions. Since \(H^s\) embeds into \(C_0\), their sums with $\Ub_\infty$ lie in a fixed tubular neighborhood of $\Gr_k(\C^d)$ for all sufficiently large indices. Applying the smooth nearest-point projection gives a sequence of smooth $\Gr_k(\C^d)$-valued maps, each equal to $\Ub_\infty$ outside a compact interval, converging to $\Ub_0$ in $H^s$. Indeed, the projection is smooth and fixes both $\Ub_0$ and $\Ub_\infty$.

Therefore, it suffices to approximate such a smooth map $\Ub$. Let $\zeta(x):=\frac{x-\ii}{x+\ii},\enspace \Ub(x):=\widetilde\Ub(\zeta(x))$. The map $\widetilde\Ub$ is smooth on the unit circle and is constant near $\zeta=1$. Let $\widetilde P:=(\mathbbm{1}_d-\widetilde\Ub)/2$. As in the proof of rational density in \cite[Appendix A]{gerard2025global}, choose
$G\in C^\infty (\mathbb{T};\mathbb{C}^{d\times k})$ such that
$$
\rank(G(\zeta))=k,\enspace \widetilde P(\zeta)=G(\zeta)\left[G(\zeta)^*G(\zeta)\right]^{-1}G(\zeta)^*,\enspace \zeta\in\mathbb{T}.
$$
Let $G_N$ be the truncated Fourier series of $G$ and replace it by $\widetilde G_N(\zeta)=G_N(\zeta)-G_N(1)+G(1),\enspace \zeta\in\mathbb{T}$. For every integer $q\geq0$, $\widetilde G_N$ converges to $G$ in $C^q$. For all large $N$, one has $\widetilde G_N(1)=G(1),\enspace \rank(\widetilde G_N(\zeta))=k,\enspace \zeta\in\mathbb{T}$. Consequently, $\widetilde\Ub_N:=\mathbbm{1}_d-2\widetilde G_N(\widetilde G_N^*\widetilde G_N)^{-1}\widetilde G_N^*$ is rational on the circle, takes values in $\Gr_k(\C^d)$, and equals $\Ub_\infty$ at $1$. Therefore, $\Ub_N=\widetilde\Ub_N\circ\zeta$ is rational on $\R$ with the prescribed background. If $E_N=\widetilde\Ub_N-\widetilde\Ub$, then $E_N(1)=0$. The chain rule and $|\zeta(x)-1|\lesssim\langle x\rangle^{-1}$ give
$$
|\partial_x^\ell(E_N\circ\zeta)(x)|_{M_d(\C)}\leq C_q\|E_N\|_{C^q}\langle x\rangle^{-1},\enspace 0\leq\ell\leq q.
$$
Taking an integer $q\geq s$ proves convergence in $H^s(\R)$.

It remains to consider $s=1/2$. Let $\rho\in C_c^\infty((-1,1))$ be nonnegative with $\int\rho=1$, set $\rho_\varepsilon(x)=\varepsilon^{-1}\rho(x/\varepsilon)$, and define $F_\varepsilon=\Ub_\infty+\rho_\varepsilon*\mathbf V_0$. Then $F_\varepsilon-\Ub_0\to0$ in $H^{1/2}$. Writing $I_{x,\varepsilon}=(x-\varepsilon,x+\varepsilon)$, one has
$$
\operatorname{dist}\big(F_\varepsilon(x),\Gr_k(\C^d)\big)^2\leq C\iint_{I_{x,\varepsilon}^2}
       \frac{|\Ub_0(y)-\Ub_0(z)|_{M_d(\C)}^2}{|y-z|^2}\,\mathrm dy\mathrm dz.
$$
The last integrand belongs to $L^1(\R^2)$. Since $|I_{x,\varepsilon}^2|=4\varepsilon^2$, absolute continuity of its integral shows that the last bound tends to zero uniformly in $x$. Thus for small $\varepsilon$, the smooth nearest-point projection $P$ onto $\Gr_k(\C^d)$ is defined at every $F_\varepsilon(x)$.

Let $\Ub_{0,\varepsilon}=P(F_\varepsilon)$. Extend $P$ smoothly to the space of Hermitian matrices with bounded derivative. By the continuity of smooth composition in $H^{1/2}$, we have $\|\Ub_{0,\varepsilon}-\Ub_0\|_{H^{1/2}}\to0$. Every derivative of $F_\varepsilon-\Ub_\infty$ belongs to $L^2\cap L^\infty$. It follows that $\Ub_{0,\varepsilon}-\Ub_\infty\in H^q$ for every integer $q\geq0$. The case $s=2$ and a diagonal choice give the required rational approximations.
\end{proof}

The next two lemmas establish convergence and isometry of the transform in the explicit formula.
\begin{lemma}\label{observation contraction}
Let $B$ be a bounded self-adjoint operator on $L^2_+(\R;\mathcal V)$. Set $A_B:=X^*+B$ with $\dom(A_B)=\dom(X^*)$. Then:
\begin{itemize}
\item[(i)]The transform
$$
\mathcal O_B F(z) := \frac{1}{2\pi \ii} I_+ \big( (A_B - z \id)^{-1} F \big), \enspace \forall z \in \C_+
$$
defines a contraction $\mathcal O_B$ on $L^2_+\left(\mathbb{R};\mathcal{V}\right)$.
\item[(ii)]The strongly continuous semigroup $W_B(\eta):=\eu^{-\ii\eta A_B},\enspace\eta\geq0$ satisfies $W_B(\eta)W_B(\eta)^*=\id$.
\item[(iii)]One has the identity
\begin{align}\label{2e2m2}
 \|\mathcal O_BF\|_{L^2}^2=\|F\|_{L^2}^2-\lim_{\eta\to\infty}\|W_B(\eta)F\|_{L^2}^2.
\end{align}
\item[(iv)]\(\ker\left(\mathcal O_B\right)\) reduces \(W_B(\eta)\), and $W_B(\eta)$ restricts to a unitary group on \(\ker\left(\mathcal O_B\right)\).
\item[(v)]If \(B_j\to B\) strongly and $\sup\{\|B_j\|_{\mathrm{op}}\}<\infty$, then \(\mathcal O_{B_j}F\rightharpoonup \mathcal O_BF\) for every $F\in L^2_+\left(\mathbb{R};\mathcal{V}\right)$.
\end{itemize}
\end{lemma}
\begin{proof}
Since
\begin{align}\label{graph and X}
\text{Im}\left\langle X^*f,f\right\rangle_{L^2}=-\frac{1}{4\pi}\left|I_+\left(f\right)\right|_{\mathcal{V}}^2\leq0,\enspace f\in\mathrm{dom}(X^*),
\end{align}
the operator $-\ii X^*$ is dissipative. For $z$ satisfying $\mathrm{Im}z>0$ and $g\in L^2_+$, the equation $(X^*-z\id)f=g$ has a unique solution $f\in\mathrm{dom}(X^*)$, given by
$$
\widehat{f}(\xi)=\ii\int_\xi^{\infty}e^{\ii z(\eta-\xi)}\widehat{g}(\eta)d\eta\;\mathbbm{1}_{\xi>0},
$$
and
$$
\|\widehat{f}\|_{L^2}\leq \int_0^{\infty}e^{-s\mathrm{Im}z}\|\widehat{g}(\cdot+s)\|_{L^2}ds\leq(\mathrm{Im}z)^{-1}\|\widehat{g}\|_{L^2},\enspace \frac{\mathrm{d}}{\mathrm{d}\xi}\widehat{f}=-\ii z\widehat{f}-\ii\widehat{g}\in L^2(\R_+).
$$
It follows that $-\ii X^*$ is a maximal dissipative operator, and $\|(X^*-z\id)^{-1}\|_{\mathrm{op}}\leq (\mathrm{Im}z)^{-1}$. For $z$ satisfying $\operatorname{Im}z>\|B\|_{\mathrm{op}}$, the operator $A_B-z\id$ is invertible. Thus $-\ii A_B$ is also maximal dissipative. For every $z\in\C_+$, identity \eqref{graph and X}, together with $A_B=X^*+B$ and $B=B^*$, gives
$$
(\operatorname{Im}z)\|f\|_{L^2}^2\leq-\operatorname{Im}\langle(A_B-z\id)f,f\rangle_{L^2}\leq\|(A_B-z\id)f\|_{L^2}\|f\|_{L^2}.
$$
It follows that
\begin{align}\label{no inverse}
	\|(A_B-z\id)f\|_{L^2}\geq(\operatorname{Im}z)\|f\|_{L^2},\enspace f\in\dom(A_B).
\end{align}
To prove the invertibility of $A_B-z\id$ for every $z\in\C_+$, choose $Y>\max\{\operatorname{Im}z,\|B\|_{\mathrm{op}}\},\enspace z_0=\operatorname{Re}z+\ii Y$. The preceding argument gives the existence of $(A_B-z_0\id)^{-1}$, and the estimate above yields
$$
\|(A_B-z_0\id)^{-1}\|_{\mathrm{op}}\leq Y^{-1},\enspace \|(z_0-z)(A_B-z_0\id)^{-1}\|_{\mathrm{op}}\leq(Y-\operatorname{Im}z)Y^{-1}<1.
$$
Hence $\id+(z_0-z)(A_B-z_0\id)^{-1}$ is invertible. The factorization $A_B-z\id=\bigl(\id+(z_0-z)(A_B-z_0\id)^{-1}\bigr)(A_B-z_0\id)\enspace\text{on }\dom(A_B)$ then proves that $A_B-z\id$ is invertible. By \eqref{no inverse},
\begin{align}\label{yujiebound}
 \|(A_B-z\id)^{-1}\|_{\mathrm{op}}\leq(\operatorname{Im}z)^{-1},\enspace z\in\C_+.
\end{align}
By the Hille--Yosida theorem, $-\ii A_B$ generates the strongly continuous contraction semigroup $W_B$. Furthermore, the adjoint semigroup \(W_B^*\) is generated by $(-\ii A_B)^*=\ii A_B^*=\ii(X+B),\enspace \dom(X)=\{f\in\dom(X^*):I_+(f)=0\}$. For $f\in\dom(X)$, the symmetry of $X+B$ yields $\frac{\mathrm d}{\mathrm d\eta}\|W_B(\eta)^*f\|_{L^2}^2=0$. The density of $\dom(X)$ and the contraction of the semigroup imply that $W_B(\eta)^*$ is an isometry. It follows that
$W_B(\eta)W_B(\eta)^*=\id$.

For $F\in\dom(A_B)$, we obtain
$$
\frac{\mathrm d}{\mathrm d\eta}\|W_B(\eta)F\|_{L^2}^2=2\operatorname{Im}\left\langle A_BW_B(\eta)F,\,W_B(\eta)F\right\rangle_{L^2}=-\frac1{2\pi}\left|I_+\bigl(W_B(\eta)F\bigr)\right|_{\mathcal V}^2.
$$
Integrating this identity over $[0,r]$ yields
\begin{align}\label{2m2e2}
    \|F\|_{L^2}^2-\|W_B(r)F\|_{L^2}^2=\frac1{2\pi}\int_0^r \left|I_+\bigl(W_B(\eta)F\bigr)\right|_{\mathcal V}^2\,\mathrm d\eta.
\end{align}
By the density of $\dom(X^*)$ and the above estimate, the map $F\mapsto I_+(W_B(\cdot)F)$ extends uniquely to a bounded linear operator from \(L^2_+\) into \(L^2(\mathbb R_+)\). The resolvent formula yields
$$
(A_B-z\id)^{-1}F=\ii\int_0^\infty\eu^{\ii\eta z}W_B(\eta)F\,\mathrm d\eta,\enspace (\mathcal O_BF)(z)=\frac1{2\pi}\int_0^\infty\eu^{\ii\eta z} I_+(W_B(\eta)F)\,\mathrm d\eta.
$$
For $F\in\dom(A_B)$, the first integral converges in the graph norm of $A_B$, so $I_+$ can be applied to it. A density argument gives the second formula for general $F\in L^2_+$. It follows that
$$
\widehat{\mathcal O_BF}(\eta)=\mathbbm{1}_{\xi\geq0}(\eta)\,I_+\bigl(W_B(\eta)F\bigr),\enspace \|\mathcal O_BF\|_{L^2}^2=\frac{1}{2\pi}\int_0^\infty\left|I_+\bigl(W_B(\eta)F\bigr)\right|_{\mathcal V}^2\,\mathrm d\eta.
$$
Letting \(r\to\infty\) in \eqref{2m2e2} yields \eqref{2e2m2}. In particular, one has
$$
\|\mathcal O_BF\|_{L^2}\leq\|F\|_{L^2},\enspace F\in L^2_+\left(\mathbb{R};\mathcal{V}\right),
$$
Therefore, $\mathcal O_B$ is a contraction.

Let $P_r=W_B(r)^*W_B(r)$. Since $W_B(r)$ is a coisometry, the orthogonal projections $P_r$ form a decreasing family and converge strongly to the projection $P_\infty$ onto $\mathcal{N}:=\bigcap_{r\geq0}\ran (W_B(r)^*)$. From \eqref{2e2m2}, one has
\begin{align}\label{2em2}
 \|\mathcal O_BF\|_{L^2}^2=\|F\|_{L^2}^2-\|P_\infty F\|_{L^2}^2\enspace\text{and}\enspace \mathrm{ker} (\mathcal{O}_B)=\mathcal{N}.
\end{align}
For any $F\in\mathcal N$ and $r,a\geq0$, write $F=W_B(r+a)^*G_{r+a}=W_B(r)^*G_{r}$. We have $W_B(a)F=W_B(r)^*G_{r+a}\in\ran (W_B(r)^*)$, $W_B(a)^*F=W_B(a)^*W_B(r)^*G_{r}=W_B(r)^*W_B(a)^*G_{r}\in\ran (W_B(r)^*)$.
Thus $\mathrm{ker} (\mathcal{O}_B)$ reduces $W_B$. On this subspace, both $W_B(a)W_B(a)^*$ and $W_B(a)^*W_B(a)$ equal the identity. Therefore, the restricted semigroup extends to a unitary group.

Finally, \eqref{yujiebound} and the strong convergence $B_j\to B$ give
$$
\|(A_{B_j}-z)^{-1}F-(A_B-z)^{-1}F\|_{L^2}\to 0,\enspace \forall\; z\in\C_+,\; F\in L^2_+(\R;\mathcal{V}).
$$
Since $\sup\{\|B_j\|_{\mathrm{op}}\}<\infty$ and $X^*(A_{B_j}-z)^{-1}F=F+(z-B_j)(A_{B_j}-z)^{-1}F$, the convergence also holds in the graph norm of $X^*$. The continuity of $I_+$ in this norm gives $\mathcal O_{B_j}F(z)\to\mathcal O_BF(z),\enspace \forall \;z\in\C_+$. Since point evaluations span a dense subspace of the dual of $L^2_+$, the contraction bound implies $\mathcal O_{B_j}F\rightharpoonup\mathcal O_BF$ in $L^2_+$.
\end{proof}

\begin{lemma}\label{OtTisometry}
Set $t\in\R$. Let $T$ and $\widetilde T$ be bounded self-adjoint operators on $L^2_+\left(\mathbb{R};\mathcal{V}\right)$. If $\widetilde T\mathcal O_{tT}=\mathcal O_{tT} T$, then $\mathcal O_{tT}$ is an isometry.
\end{lemma}
\begin{proof}
Set $\mathcal N_t=\ker(\mathcal O_{tT})$. The intertwining relation gives $T\mathcal N_t\subset\mathcal N_t$. Since $T=T^*$, the subspace $\mathcal N_t$ reduces $T$. Lemma \ref{observation contraction} shows that $\mathcal N_t$ also reduces $W_{tT}(\eta):=e^{-\ii\eta A}$ for $\eta\geq0$, where $A=X^*+tT$. Hence the orthogonal projection $P$ onto this subspace satisfies $PW_{tT}(\eta)=W_{tT}(\eta)P,\enspace P\dom(A)\subset\dom(A),\enspace AP=PA\;\;\mathrm{on}\;\;\dom(A)$. Since $\dom(A)=\dom(X^*)$, we obtain $PX^*f=P(A-tT)f=(A-tT)Pf=X^*Pf,\enspace f\in \dom(X^*)$. Therefore, for $S(\eta):=\eu^{-\ii\eta X^*}$, one has
$$
\frac{\mathrm d}{\mathrm dr}\bigl(S(\eta-r)PS(r)f\bigr)=\ii S(\eta-r)X^*PS(r)f-\ii S(\eta-r)PS(r)X^*f=0,\enspace 0<r<\eta.
$$
Thus $S(\eta-r)PS(r)f$ is constant on $[0,\eta]$. Evaluating at $r=0$ and $r=\eta$ gives $S(\eta)P=PS(\eta)\;\;\mathrm{on}\;\;\dom(X^*)$. Since $\dom(X^*)$ is dense and both $P$ and $S(\eta)$ are bounded, this identity extends to all of $L^2_+(\R;\mathcal V)$. Hence $\mathcal N_t$ reduces $S(\eta)=\eu^{-\ii\eta X^*}$.

Lemma \ref{observation contraction} also shows that $W_{tT}$ is unitary on $\mathcal N_t$. For every $f\in\mathcal N_t\cap\dom(X^*)$, one has
$$
0=\left.\frac{\mathrm d}{\mathrm d\eta}\|W_{tT}(\eta)f\|_{L^2}^2\right|_{\eta=0}=2\operatorname{Im}\langle Af,f\rangle_{L^2}=-\frac{1}{2\pi}|I_+(f)|_{\mathcal V}^2.
$$
It follows that $I_+(f)=0,\enspace f\in\mathcal N_t\cap\dom(X^*)$. Since $\mathcal N_t$ reduces $\eu^{-\ii\eta X^*}$, we have $\eu^{-\ii\eta X^*}\mathcal N_t\subset \mathcal N_t$. Thus $\eu^{-\ii\eta X^*}$ preserves $\mathcal N_t\cap\dom(X^*)$. It follows that
$$
\widehat f(\eta)=I_+(\eu^{-\ii\eta X^*}f)=0,\enspace \eta\geq0,\enspace\text{and}\enspace \mathcal N_t\cap\dom(X^*)=\{0\}.
$$
Since the orthogonal projection $P$ satisfies $P\dom(X^*)\subset\dom(X^*)$, the space $\mathcal N_t\cap\dom(X^*)$ is dense in $\mathcal N_t$. It follows that $\mathcal N_t=\{0\}$. From \eqref{2em2}, we have
$\|\mathcal O_{tT}F\|_{L^2}=\|F\|_{L^2},\enspace F\in L^2_+(\R;\mathcal V)$.

\end{proof}
We prove the matrix version of Theorem \ref{global energy}.
\begin{proposition}\label{global solution matrix}
Equation \eqref{HWMm} is globally well-posed in $H^{1/2}_\bullet(\R;\Gr_k(\C^d))$. For every initial datum $\Ub_0=\Ub_\infty+\mathbf V_0\in H^{1/2}_\bullet(\R;\Gr_k(\C^d))$, the explicit formula \eqref{explicit formula} defines the corresponding global-in-time solution $\Ub(t)=\Phi_t(\Ub_0)=\Ub_\infty+\mathbf V(t),\enspace \mathbf V\in C(\R;H^{1/2}(\R;M_d(\C)))$, which solves \eqref{HWMm} in the distributional sense and has the following properties.
\begin{itemize}
	\item[(i)] The following functionals are conserved.
	$$
\|\mathbf V(t)\|_{L^2}=\|\mathbf V_0\|_{L^2},\enspace \|\mathbf V(t)\|_{\dot H^{1/2}}=\|\mathbf V_0\|_{\dot H^{1/2}},\enspace t\in\R.
$$
	\item[(ii)] The discrete spectrum of $T_{\Ub(t)}$ is preserved.
	\item[(iii)] The flow maps satisfy
	$\Phi_0=\id,\enspace \Phi_{t+\tau}=\Phi_t\circ\Phi_\tau,\enspace t,\tau\in\R$,
	and agree with the global classical rational flow on $\mathcal{R}at(\R;\Gr_k(\C^d))$.
\end{itemize}
\end{proposition}
\begin{proof}
By the density established in Lemma \ref{rational fixed background}, let $\Ub_{0,j}=\Ub_{\infty}+\mathbf V_{0,j}\in\mathcal{R}at(\R;\Gr_k(\C^d))$ be any sequence such that $\|\mathbf V_{0,j}-\mathbf V_0\|_{H^{1/2}}\to0$. Denote the corresponding global classical rational solutions by $\Ub_j=\Ub_\infty+\mathbf V_j$, and denote the corresponding Toeplitz operator $T_{\Ub_{0,j}}$ by $T_j$. The proof of the explicit formula in \cite{gerard2025global} gives
\begin{align}\label{rational explicit}
 \Pi_+\mathbf V_j(t)=\mathcal O_{tT_j}\Pi_+\mathbf V_{0,j}\enspace \text{and}\enspace T_{\Ub_j(t)}\mathcal O_{tT_j}=\mathcal O_{tT_j}T_j\enspace\text{on }L^2_+(\R;\mathcal{V}),
\end{align}
where $\mathcal O_{tT_j}$ is unitary. In particular,
$$
\|\mathbf V_j(t)\|_{L^2}=\|\mathbf V_{0,j}\|_{L^2},\enspace\|\mathbf V_j(t)\|_{\dot H^{1/2}}=\|\mathbf V_{0,j}\|_{\dot H^{1/2}}.
$$
Thus $\mathbf V_j(t)$ is bounded in $H^{1/2}$ uniformly in $j$ and $t$.

The initial convergence gives $\mathbf V_{0,j}\to\mathbf V_0$ in $L^2$. Together with $\|\Ub_{0,j}(x)\|_{\mathrm{op}}=1$, this implies $T_j\to T_0$ strongly, where $T_0=T_{\Ub_0}$. For every $t_j\to t\in\R$, we thus have $t_jT_j\to tT_0$ strongly, and Lemma \ref{observation contraction} yields
$$
\Pi_+\mathbf V_j(t_j)=\mathcal O_{t_jT_j}\Pi_+\mathbf V_{0,j}\rightharpoonup\mathcal O_{tT_0}\Pi_+\mathbf V_0\enspace\text{in }L^2(\R;\mathcal{V}).
$$
Set $\mathbf V(t):=\mathcal O_{tT_0}\Pi_+\mathbf V_0+(\mathcal O_{tT_0}\Pi_+\mathbf V_0)^*$. This weak $L^2$ limit depends only on $\Ub_0$ and $t$. The uniform $H^{1/2}$ bound therefore yields
\begin{align}\label{weak convergence}
 \mathbf V_j(t_j)\rightharpoonup\mathbf V(t)\enspace\text{in }H^{1/2}(\R;\mathcal{V}),
\end{align}
and the compact embedding gives
\begin{align}\label{local convergence}
\mathbf V_j(t_j)\to\mathbf V(t)\enspace\text{in }L^2_{\mathrm{loc}}(\R;\mathcal{V}).
\end{align}
A subsequence of $\mathbf V_j(t_j)$ converges almost everywhere to $\mathbf V(t)$. Since $\Gr_k(\C^d)$ is closed, $\Ub(t):=\Ub_\infty+\mathbf V(t)$ takes values in this manifold. In addition, weak lower semicontinuity gives
\begin{align}\label{preliminary energy inequality}
\|\mathbf V(t)\|_{\dot H^{1/2}}^2\le\liminf_{j\to\infty}\|\mathbf V_j(t_j)\|_{\dot H^{1/2}}^2 =\liminf_{j\to\infty}\|\mathbf V_{0,j}\|_{\dot H^{1/2}}^2=\|\mathbf V_0\|_{\dot H^{1/2}}^2.
\end{align}

The local convergence in \eqref{local convergence} implies $T_{\Ub_j(t_j)}\to T_{\Ub(t)}\enspace\text{strongly on }L^2_+(\R;\mathcal{V})$, where the density of bounded, compactly supported functions in \(L^2\) is used. Taking the weak limit in the second identity of \eqref{rational explicit}, with the self-adjoint Toeplitz operator acting on the test function, yields
\begin{align}\label{intertwining identity}
 T_{\Ub(t)}\mathcal O_{tT_0}=\mathcal O_{tT_0}T_0\enspace\text{on }L^2_+(\R;\mathcal{V}).
\end{align}
Lemma \ref{OtTisometry} shows that $\mathcal O_{tT_0}$ is an isometry, and thus
$\|\mathbf V(t)\|_{L^2}=\|\mathbf V_0\|_{L^2}$.

Let $P_t$ be the orthogonal projection onto the closed range of the isometry $\mathcal O_{tT_0}$. By \eqref{intertwining identity}, this range is invariant under $T_{\Ub(t)}$ and hence reducing, since $T_{\Ub(t)}$ is self-adjoint. Moreover, \eqref{TandK} and \eqref{intertwining identity} give $K_{\Ub(t)}\mathcal O_{tT_0}=\mathcal O_{tT_0}K_{\Ub_0}\enspace\text{on }L^2_+(\R;\mathcal{V})$. Let \(\{e_j\}\) be an orthonormal basis of $L^2_+(\mathbb{R};\C^d)$. Then \(\left\{\mathcal O_{tT_0} e_j\right\}\) is an orthonormal basis of \(\ran(\mathcal O_{tT_0})\). Therefore,
\begin{align*}
\|\mathbf V(t)\|_{\dot H^{1/2}}^2=&4\pi\Tr \left(K_{\Ub(t)}\right)\geq4\pi\Tr\big(K_{\Ub(t)}|_{\ran(\mathcal O_{tT_0})}\big)=4\pi\sum \langle K_{\Ub(t)}\mathcal O_{tT_0} e_j,\mathcal O_{tT_0} e_j\rangle\\
 =&4\pi\sum \langle \mathcal O_{tT_0}K_{\Ub_0} e_j,\mathcal O_{tT_0} e_j\rangle=4\pi\sum \langle K_{\Ub_0}e_j,e_j\rangle=4\pi\Tr \left(K_{\Ub_0}\right)=\|\mathbf V_0\|_{\dot H^{1/2}}^2.
\end{align*}
Combining this with \eqref{preliminary energy inequality} yields
$$
\|\mathbf V(t)\|_{\dot H^{1/2}}^2=\|\mathbf V_0\|_{\dot H^{1/2}}^2,\enspace \Tr \left(K_{\Ub(t)}\right)=\Tr\big(K_{\Ub(t)}|_{\ran(\mathcal O_{tT_0})}\big).
$$
It follows that
\begin{align}\label{1019}
K_{\Ub(t)}|_{\ran(\mathcal O_{tT_0})^\perp}=0\;\text{ and }\; (T_{\Ub(t)}|_{\ran(\mathcal O_{tT_0})^\perp})^2=\id.
\end{align}
To prove preservation of the discrete spectrum, let $v\in(-1,1)$. If $T_0f=vf$ with $f\neq0$, then $T_{\Ub(t)}\mathcal O_{tT_0} f=\mathcal O_{tT_0}T_0f=v\mathcal O_{tT_0} f$, and $\mathcal O_{tT_0} f\neq0$. Conversely, suppose that $T_{\Ub(t)}g=vg$ with $g\neq0$. Decompose
$$
g=\mathcal O_{tT_0} f+h,\enspace f=\mathcal O_{tT_0}^*g,\enspace h=(\id-P_t)g\in\ran(\mathcal O_{tT_0})^\perp.
$$
Since $\ran(\mathcal O_{tT_0})$ reduces $T_{\Ub(t)}$, the identity $0=(T_{\Ub(t)}-v\id)g=\mathcal O_{tT_0}(T_0-v\id)f+(T_{\Ub(t)}-v\id)h$ is an orthogonal decomposition into vectors in $\ran(\mathcal O_{tT_0})$ and $\ran(\mathcal O_{tT_0})^\perp$, respectively, so both terms vanish. The injectivity of $\mathcal O_{tT_0}$ gives $T_0f=vf$, while $T_{\Ub(t)}h=vh$. By \eqref{1019}, we obtain $h=T_{\Ub(t)}^2h=v^2h$. Since $|v|<1$, this implies $h=0$. Consequently, $g=\mathcal O_{tT_0} f$ and $f\neq0$. We have proved $\ker(T_{\Ub(t)}-v\id)=\mathcal O_{tT_0}\ker(T_0-v\id),\enspace v\in(-1,1)$. Therefore, we have $\sigma_{\mathrm d}(T_{\Ub(t)})=\sigma_{\mathrm d}(T_{\Ub_0})$.

By \eqref{weak convergence} and conservation of the $L^2$ norm and energy,
$$
 \|\mathbf V_j(t_j)\|_{\dot H^a}=\|\mathbf V_{0,j}\|_{\dot H^a}
 \longrightarrow\|\mathbf V_0\|_{\dot H^a}=\|\mathbf V(t)\|_{\dot H^a},\enspace a=0,\tfrac12.
$$
Thus $\mathbf V_j(t_j)\to\mathbf V(t)$ strongly in $H^{1/2}$. As $s\to t$, Lemma \ref{observation contraction} gives
$$
 \mathcal O_{sT_0}\Pi_+\mathbf V_0\rightharpoonup
 \mathcal O_{tT_0}\Pi_+\mathbf V_0\enspace\text{in }L^2.
$$
The isometry property upgrades this to strong convergence in $L^2$. The uniform $H^{1/2}$ bound then gives $\mathbf V(s)\rightharpoonup\mathbf V(t)$ in $H^{1/2}$, and conservation of both norms yields strong convergence. Hence $\mathbf V\in C(\R;H^{1/2}(\R;M_d(\C)))$. A contradiction argument using a convergent subsequence of times gives
$\mathbf V_j\to\mathbf V$ in $C([-T,T];H^{1/2}(\R;M_d(\C)))$ for every $T>0$.
It follows from $\mathcal O_0=\id$ that $\Ub(0)=\Ub_0$. Since the rational approximating sequence was arbitrary, $\Ub$ is the global-in-time solution in the sense of Definition \ref{def global solution}. For rational initial data, the constant approximating sequence shows that this solution agrees with the global classical rational solution.

For $\Psi\in C_c^\infty(\R_t\times\R_x;M_d(\C))$, define the nonlinear term by
$$
\int_{\R}\Tr\left([\Ub,|D|\Ub]\Psi\right)\,\mathrm dx:=\int_{\R}\Tr\big(|D|^{1/2}\mathbf V\,|D|^{1/2}[\Psi,\Ub]\big)\,\mathrm dx,
$$
where the right-hand side is well-defined. The strong convergence allows us to pass to the limit in the equations satisfied by rational solutions, yielding
$$
\int_{\R}\int_{\R}\Tr(\mathbf V\partial_t\Psi)\,\mathrm dx\mathrm dt=\frac{\ii}{2}\int_{\R}\int_{\R} \Tr\big(|D|^{1/2}\mathbf V\,|D|^{1/2}[\Psi,\Ub]\big)\,\mathrm dx\mathrm dt.
$$
Thus $\Ub$ solves \eqref{HWMm} in the distributional sense.

For rational $\Ub_{0,j}$, the classical solutions
$U_1(t)=\Phi_{t+\tau}(\Ub_{0,j})$ and
$U_2(t)=\Phi_t(\Phi_\tau(\Ub_{0,j}))$ have the same initial value $\Phi_\tau(\Ub_{0,j})$. Classical uniqueness gives
$\Phi_{t+\tau}(\Ub_{0,j})=\Phi_t(\Phi_\tau(\Ub_{0,j}))$.
Passing to the limit yields $\Phi_{t+\tau}(\Ub_0)=\Phi_t(\Phi_\tau(\Ub_0))$; together with $\Phi_0=\id$, this is the group property.

Finally, we prove continuous dependence on initial data. Let $\Ub_{0}^{(n)}=\Ub_{\infty}+\mathbf V_{0}^{(n)}\in H^{1/2}_\bullet(\R;\Gr_k(\C^d))$ be any sequence such that $\|\mathbf V_{0}^{(n)}-\mathbf V_0\|_{H^{1/2}}\to0$. Together with the uniform bound $\|\Ub_0^{(n)}(x)\|_{\mathrm{op}}=1$, this implies $T^{(n)}\to T_0$ strongly, where $T^{(n)}=T_{\Ub_0^{(n)}}$ and $T_0=T_{\Ub_0}$. For every $t_n\to t\in\R$, we thus have $t_nT^{(n)}\to tT_0$ strongly, and Lemma \ref{observation contraction} yields
$$
\Pi_+\mathbf V^{(n)}(t_n)=\mathcal O_{t_nT^{(n)}}\Pi_+\mathbf V_0^{(n)}\rightharpoonup\mathcal O_{tT_0}\Pi_+\mathbf V_0=\Pi_+\mathbf V(t)\enspace\text{in }L^2_+(\R;M_d(\C)),
$$
where $\mathbf V^{(n)}(t)=\Phi_t(\Ub_0^{(n)})-\Ub_\infty$. The conservation laws imply that $\mathbf V^{(n)}(t_n)$ is bounded in $H^{1/2}$. It follows that $\mathbf V^{(n)}(t_n)\rightharpoonup\mathbf V(t)\enspace\text{in }H^{1/2}$. The same conservation laws yield
\begin{align*}
 &\|\mathbf V^{(n)}(t_n)\|_{L^2}^2+\|\mathbf V^{(n)}(t_n)\|_{\dot H^{1/2}}^2=\|\mathbf V_{0}^{(n)}\|_{L^2}^2+\|\mathbf V_{0}^{(n)}\|_{\dot H^{1/2}}^2\\
 &\to\|\mathbf V_0\|_{L^2}^2+\|\mathbf V_0\|_{\dot H^{1/2}}^2=\|\mathbf V(t)\|_{L^2}^2+\|\mathbf V(t)\|_{\dot H^{1/2}}^2.
\end{align*}
Weak convergence and convergence of these norms imply
\begin{align}\label{s l y}
\|\Phi_{t_n}(\Ub_{0}^{(n)})-\Phi_t(\Ub_0)\|_{H^{1/2}}\to0.
\end{align}
For uniform convergence on $[-T,T]$, suppose that $\|\Phi_{t_n}(\Ub_{0}^{(n)})-\Phi_{t_n}(\Ub_0)\|_{H^{1/2}}\geq\varepsilon_0>0,\enspace t_n\in[-T,T]$. A subsequence satisfies $t_n\to t\in[-T,T]$. It follows from \eqref{s l y} that
$$
\|\Phi_{t_n}(\Ub_{0}^{(n)})-\Phi_t(\Ub_0)\|_{H^{1/2}}\to0\enspace \text{and} \enspace\|\Phi_{t_n}(\Ub_{0})-\Phi_t(\Ub_0)\|_{H^{1/2}}\to0.
$$
This contradicts the preceding lower bound.
\end{proof}
Taking $d=2$, $k=1$ in Proposition \ref{global solution matrix} proves Theorem \ref{global energy}.

\section{Proof of Theorem \ref{thm illposedness}: Ill-posedness below $H^{1/2}_\bullet$}\label{sec illposedness}
Write $\mathbf d_n=\mathbf u_{0,n}-f(x)\mathbf e_3$. For $0\le y<1$, $2y^{2n}/(1+y^{4n})\le2y^{2n}$ and $|(y^{4n}-1)/(1+y^{4n})+1|\le2y^{4n}$.
For $y>1$ the corresponding bounds are $2y^{-2n}$ and $2y^{-4n}$. Integration on the two half-lines gives
$\|\mathbf d_n\|_{L^1}\le C/n$ and $\operatorname{TV}(\mathbf d_n)\le C$, where $\operatorname{TV}$ denotes total variation. For the second estimate, the first component of $\mathbf u_{0,n}$ increases from $0$ to $1$ and decreases to $0$ on the positive half-line, while its third component increases from $-1$ to $1$. The limiting step has bounded variation as well. Hence $|\widehat{\mathbf d_n}(\xi)|\le C\min\{n^{-1},|\xi|^{-1}\}$ for $\xi\ne0$, so
$$
\|\mathbf d_n\|_{H^s}^2\le C_s\left[n^{-2}\int_0^n(1+\xi^2)^s\,d\xi+\int_n^\infty\xi^{2s-2}\,d\xi\right]\le C_s n^{2s-1},\enspace 0\le s<1/2.
$$
For $s<0$, the convergence follows from $\|\mathbf d_n\|_{H^s}\le\|\mathbf d_n\|_{L^2}$. In addition, $\mathbbm1_I\in H^s$ for every $s<1/2$, as follows directly from $\widehat{\mathbbm1_I}(\xi)=2\sin\xi/\xi$. This proves \eqref{sharp initial convergence}.

To pass to the limit in the explicit formula, set
$$
\mathbf U_{0,n}=\mathbf u_{0,n}\cdot\bm\sigma,\enspace \mathbf U_*=f\sigma_3,\enspace \mathbf V_{0,n}=\mathbf U_{0,n}-\sigma_3,\enspace \mathbf V_*=\mathbf U_*-\sigma_3.
$$
The pointwise convergence away from $\{-1,1\}$ and the uniform bound $\sup\|\mathbf U_{0,n}(x)\|_{\mathrm{op}}=1$ imply
$$
T_{\mathbf U_{0,n}}\to T_{\mathbf U_*}\enspace\text{strongly on }L^2_+,\enspace \Pi_+\mathbf V_{0,n}\to\Pi_+\mathbf V_*\enspace\text{in }L^2_+.
$$
Indeed, dominated convergence applies to $(\mathbf U_{0,n}-\mathbf U_*)F$ for each $F\in L^2_+$. By Lemma \ref{observation contraction}, we have
\begin{align}\label{sharp resolvent limit}
 \Pi_+\bigl(\mathbf U_n(t)-\sigma_3\bigr)=\mathcal O_{tT_{\mathbf U_{0,n}}}\Pi_+\mathbf V_{0,n}\rightharpoonup\mathcal O_{tT_{\mathbf U_*}}\Pi_+\mathbf V_*\enspace\text{in }L^2_+.
\end{align}

It remains to evaluate the diagonal scalar resolvents. We recall the following scalar consequence of G\'erard's zero dispersion formula \cite[Lemma 6 and Corollary 5]{gerard2024zero}. If $q\in L^2\cap L^\infty$ is real-valued, define $Z_q(\tau)\in L^2(\mathbb R;\mathbb R)$ by the identity
$$
\Pi_+Z_q(\tau,z)=\frac{1}{2\pi\ii}I_+\bigl[(X^*+2\tau T_q-z\id)^{-1}\Pi_+q\bigr],\enspace z\in \C_+.
$$
For $\tau\ne0$, it satisfies
$$
2\tau\partial_xZ_q(\tau)=1-\mu_{2\tau q},\enspace\int_{\mathbb R}\varphi\,d\mu_{2\tau q} =\int_{\mathbb R}\varphi\bigl(y+2\tau q(y)\bigr)\,dy,
$$
where $\varphi\in C_c^\infty(\mathbb R)$. Taking $q=\mathbbm1_I$ and writing $F_\tau=Z_{\mathbbm1_I}(\tau)$, the measure $\mu_{2\tau q}$ has density $1-\mathbbm1_I+\mathbbm1_{I+2\tau}$.
Consequently,
$$
2\tau\partial_xF_\tau=\mathbbm1_I-\mathbbm1_{I+2\tau},\enspace F_\tau(x)=\frac{1}{2\tau}\int_{x-2\tau}^{x}\mathbbm1_I(y)\,dy.
$$
For $\tau<0$, the integral is oriented; the condition $F_\tau\in L^2$ fixes the integration constant.

Since $T_f=\id-2T_{\mathbbm1_I}$, the upper diagonal entry in \eqref{sharp resolvent limit} is
$$
-\frac{2}{2\pi\ii}I_+\bigl[(X^*-2tT_{\mathbbm1_I}-(z-t)\id)^{-1}\Pi_+\mathbbm1_I\bigr]=-2\Pi_+F_{-t}(z-t).
$$
The lower diagonal entry is $2\Pi_+F_t(z+t)$, and both off-diagonal entries vanish. Reconstructing the Hermitian boundary values gives
\begin{align*}
 \mathbf U_n(t)-\sigma_3\rightharpoonup
 \begin{pmatrix}-2F_{-t}(x-t)&0\\0&2F_t(x+t)\end{pmatrix}.
\end{align*}
The scalar formula yields
$$
F_{-t}(x-t)=F_t(x+t)
 =\frac{1}{2|t|}\int_{x-|t|}^{x+|t|}\mathbbm1_I(y)\,dy,
$$
which proves \eqref{sharp weak limit}. Let $t\neq0$. The overlap $|I\cap[x-|t|,x+|t|]|$ takes values strictly between $0$ and $2|t|$ on a set of positive measure. Thus $|h_t|<1$ there, and $h_t\mathbf e_3$ is not sphere-valued.

Finally, for $0<|t|<1$, the limiting profile is
\begin{align*}
 h_t(x)=
 \begin{cases}
  1,&|x|\ge1+|t|,\\
  -1,&|x|\le1-|t|,\\
  (|x|-1)/|t|,&1-|t|<|x|<1+|t|.
 \end{cases}
\end{align*}
Direct integration gives
$$
\|(h_t-1)\mathbf e_3\|_{L^2}^2=8-\frac{8|t|}{3},\enspace \|f(\cdot)\mathbf{e}_3-\mathbf e_3\|_{L^2}^2=8.
$$
Conservation of the $L^2$ norm for each rational solution and \eqref{sharp weak limit} therefore imply
$$
\lim_{n\to\infty}\|\mathbf u_n(t)-h_t\mathbf e_3\|_{L^2}^2=8-\|(h_t-1)\mathbf e_3\|_{L^2}^2=\frac{8|t|}{3}.
$$
This proves \eqref{sharp defect} and Theorem \ref{thm illposedness}.

\section{Eigenfunctions of the Toeplitz operator}\label{section eigenfunctions}
\begin{proposition}[Eigenfunctions of the Toeplitz operator belong to the domain of \(X^*\)]\label{varphiin}
	Let $\Ub_0=\Ub_\infty+\mathbf{V}_0$ satisfy the conditions \eqref{stronginitial 1} and \eqref{stronginitial 2} with $s=1/2$, and assume that $\mathbf V_0\in C_0(\R;M_d(\C))$. If $\varphi$ is an eigenfunction of the Toeplitz operator $T_{\Ub_0}$ corresponding to an eigenvalue $v\in(-1,1)$, then $\varphi\in\mathrm{dom}(X^*)$.
\end{proposition}
\begin{proof}
	The eigenvalue equation gives
	\begin{align}\label{tzfc0}
		(v \id - \Ub_\infty)\varphi = \Pi_+(\mathbf{V}_0 \varphi).
	\end{align}
	The matrix $v\id-\Ub_\infty$ is invertible because $\Ub_\infty^2=\id$ and $|v|<1$. Fourier transformation of \eqref{tzfc0} yields
	\begin{align}\label{tzfc1}
		\widehat{\varphi}(\xi) = (v \id - \Ub_\infty)^{-1} \widehat{\mathbf{V}_0\varphi}(\xi)=(v \id - \Ub_\infty)^{-1} \frac{1}{2\pi} \int_0^\infty \widehat{\mathbf{V}}_0(\xi - \zeta) \widehat{\varphi}(\zeta) \, \mathrm{d}\zeta,\enspace \forall\; {\xi>0}.
	\end{align}
	Since $\mathbf{V}_0\varphi\in L^1(\mathbb{R})$, the right-hand side of \eqref{tzfc1} is continuous on $\mathbb{R}$, and thus
	\begin{align}\label{lianxu hds}
		I_+(\varphi) = \widehat{\varphi}(0^+) = (v \id - \Ub_\infty)^{-1} \widehat{\mathbf{V}_0\varphi}(0) = (v \id - \Ub_\infty)^{-1} \int_{\R} \mathbf{V}_0 \varphi \, \mathrm{d}x.
	\end{align}
	For $\eta>0$, define $\tau_\eta\psi(\xi)=[\psi(\xi+\eta)-\psi(\xi)]/\eta$ on $L^2(\R_+;\mathcal V)$.
	It suffices to prove convergence of $\tau_\eta\widehat{\varphi}$ in $L^2(\mathbb R_+;\mathcal V)$ as $\eta\to0$.

	Let $\mathbf{V}\in \left(L^{\infty}\cap L^{2}\right)(\mathbb{R};M_d(\mathbb{C}))$. The convolution operator on $L^2(\R_+;\mathcal V)$ defined by
	$$
A_{\mathbf V}\psi(\xi):= (v \id - \Ub_\infty)^{-1} \frac{1}{2\pi} \int_0^\infty \widehat{\mathbf{V}}(\xi - \zeta) \psi(\zeta) \, \mathrm{d}\zeta
$$
	satisfies $\|A_{\mathbf{V}}\|_{\text{op}}\leq C\|\mathbf{V}\|_{L^\infty(\R)}$, where the constant $C$ depends only on $\Ub_\infty$ and $v$. Therefore, \eqref{tzfc1} can be written as
	\begin{align}\label{tzfc2}
		\widehat{\varphi}(\xi)=A_{\mathbf{V}_0}\widehat{\varphi}(\xi),\enspace \forall\;\xi> 0.
	\end{align}
	Let $R>0$ and $\mathbf{V}_{0, R}(x):=\mathbbm{1}_{|x|\leq R}\mathbf{V}_0(x)$. Then \eqref{tzfc2} can be written as
	\begin{align}\label{tzfc3}
		\widehat{\varphi}(\xi) - A_{\mathbf{V}_0 - \mathbf{V}_{0, R}}\widehat{\varphi}(\xi) = A_{\mathbf{V}_{0, R}}\widehat{\varphi}(\xi),\enspace \forall \xi>0.
	\end{align}
	Applying $\tau_\eta$ to \eqref{tzfc3} yields
	$$
(\id - A_{\mathbf{V}_0 - \mathbf{V}_{0, R}})\tau_\eta \widehat{\varphi}= \tau_\eta A_{\mathbf{V}_{0, R}}\widehat{\varphi} + [\tau_\eta, A_{\mathbf{V}_0 - \mathbf{V}_{0, R}}]\widehat{\varphi}.
$$
	Since $\mathbf V_0\in C_0$, the bound $\|A_{\mathbf V_0-\mathbf V_{0,R}}\|\le C\|\mathbf V_0\|_{L^\infty(|x|>R)}\to0$ allows us to fix $R$ for which $\id-A_{\mathbf V_0-\mathbf V_{0,R}}$ is invertible.

Direct calculation shows that $\tau_\eta A_{\mathbf{V}_{0, R}}\widehat{\varphi}$ converges in $L^2(\mathbb{R}_+;\mathcal{V})$ to $-\ii A_{x \mathbf{V}_{0,R}}\widehat{\varphi}$ as $\eta\to 0$. By continuity of $\widehat\varphi$ on $[0,\infty)$, $[\tau_\eta, A_{\mathbf{V}_0 - \mathbf{V}_{0, R}}]\widehat{\varphi}$ converges in $L^2(\mathbb{R}_+;\mathcal{V})$ to $\left(v \id - \Ub_\infty\right)^{-1} \frac{1}{2\pi} \left(\widehat{\mathbf{V}_0}- \widehat{\mathbf{V}_{0, R}}\right)\cdot\widehat{\varphi}(0^+)$ as $\eta\to 0$. Therefore, $\tau_\eta \widehat{\varphi}$ converges in $L^2(\mathbb{R}_+;\mathcal{V})$ to
	$$
\left(\id - A_{\mathbf{V}_0 - \mathbf{V}_{0, R}}\right)^{-1}\left( -\ii A_{x \mathbf{V}_{0,R}}\widehat{\varphi} + \left(v \id - \Ub_\infty\right)^{-1} \frac{1}{2\pi} \left(\widehat{\mathbf{V}_0}- \widehat{\mathbf{V}_{0, R}}\right)\cdot\widehat{\varphi}(0^+) \right)
$$
	as $\eta\to0$, proving $\varphi\in\dom(X^*)$.
\end{proof}
\begin{remark}
	Since $\varphi\in\dom(X^*)$ in Proposition \ref{varphiin}, by \cite[Lemma 3.2]{gerard2025global}, we have
	\begin{align}\label{pre wu}
		\left(T_{\Ub_0} - v \id\right)X^*\varphi =\left[T_{\Ub_0},X^*\right]\varphi= -\frac{\ii}{2\pi} \Pi_+ \mathbf{V}_0\cdot I_+(\varphi).
	\end{align}
\end{remark}
\begin{corollary}\label{hds}
	Under the assumptions of Proposition \ref{varphiin}, we have
	$$
\left| \int_{\R} \mathbf{V}_0(x) \varphi(x) \, \mathrm{d}x \right|_{\mathcal{V}}^2=\left(1-v^2\right)\left|I_+(\varphi)\right|_{\mathcal{V}}^2.
$$
\end{corollary}

\begin{proof}
Set $a=I_+(\varphi)$ and $b=\int_\R\mathbf V_0(x)\varphi(x)\,\mathrm dx$.
Taking the inner product of \eqref{pre wu} with $\varphi$ gives
$$
0=\langle\Pi_+\mathbf V_0a,\varphi\rangle_{L^2}
 =\left\langle a,\int_\R\Pi_-\mathbf V_0(x)\varphi(x)\,\mathrm dx\right\rangle_{\mathcal V}
 =\langle a,b\rangle_{\mathcal V}.
$$
By \eqref{lianxu hds}, $b=(v\id-\Ub_\infty)a$, so
$\langle a,\Ub_\infty a\rangle_{\mathcal V}=v|a|_{\mathcal V}^2$.
Since $\Ub_\infty^*=\Ub_\infty$ and $\Ub_\infty^2=\id$,
$$
|b|_{\mathcal V}^2=v^2|a|_{\mathcal V}^2
 -2v\operatorname{Re}\langle a,\Ub_\infty a\rangle_{\mathcal V}
 +|\Ub_\infty a|_{\mathcal V}^2=(1-v^2)|a|_{\mathcal V}^2.
$$
\end{proof}
\begin{proposition}[Distorted Plancherel formula]\label{Plancherel}
	Let $\Ub_0=\Ub_\infty+\mathbf{V}_0$ satisfy the conditions \eqref{stronginitial 1} and \eqref{stronginitial 2} with $s=1/2$, and assume that $\mathbf V_0\in C_0(\R;M_d(\C))$. For $\mathcal V=\C^d$, suppose that $T_{\Ub_0}$ has simple discrete spectrum $(v_n)_{1\le n\le\omega}$, where $\omega\in\N^*\cup\{\infty\}$. Then
	\begin{align}\label{DPF}
		\left\|\Pi_+\mathbf{V}_0\right\|_{L^2}^2=\sum_{n= 1}^{\omega}\left(1-v_n^2\right)\left|I_+(\varphi_n)\right|_{\mathcal{V}}^2,
	\end{align}
	where $(\varphi_n)_{1\leq n\leq\omega}\subset L^2_+(\mathbb{R};\mathbb{C}^d)$ is the family of normalized eigenfunctions corresponding to $(v_n)_{1\le n\le\omega}$.
\end{proposition}

\begin{proof}
	The family $(\varphi_n)_{1\leq n\leq\omega}$ forms an orthonormal basis of the closed subspace $\mathfrak{H}_1(\C^d)$ in \eqref{subspace}. Let $e_1,\ldots,e_d$ be the standard basis of $\C^d$.
	Applying the proof of \cite[Proposition 4.2]{gerard2025global} to the case $\mathcal{V}=\mathbb{C}^d$ shows that for each $j=1,2,\dots,d$, the function $\Pi_+ \mathbf{V}_0\cdot e_j$ belongs to $\mathfrak{H}_1(\C^d)$. Parseval's identity in $\mathfrak H_1(\C^d)$ gives
	$$
\|\Pi_+ \mathbf{V}_0\|_{L^2}^2 = \sum_{j=1}^d \|\Pi_+ \mathbf{V}_0\cdot e_j\|_{L^2}^2 = \sum_{j=1}^d \sum_{n=1}^{\omega} |\langle \Pi_+ \mathbf{V}_0\cdot e_j, \varphi_n\rangle_{L^2}|^2,
$$
	where
	\begin{align}
		\sum_{j=1}^d|\langle \Pi_+ \mathbf{V}_0\cdot e_j, \varphi_n\rangle_{L^2}|^2=& \sum_{j=1}^d |\langle \mathbf{V}_0\cdot e_j, \varphi_n\rangle_{L^2}|^2
		=\sum_{j=1}^d |e_j^* \int_{\R}  \mathbf{V}_0(x)\varphi_n(x) \, \mathrm{d}x|^2 \label{identity}\\
		=&\left| \int_{\mathbb{R}} \mathbf{V}_0(x) \varphi_n(x) \, \mathrm{d}x \right|_{\mathcal{V}}^2=(1-v_n^2)|I_+(\varphi_n)|_{\mathcal{V}}^2.\notag
	\end{align}
	Summing and applying Corollary \ref{hds} gives \eqref{DPF}.
\end{proof}
\begin{remark}
	For finite or countable discrete spectrum, define as in \cite{gerard2025global}
	$$
\Phi_{n,j}:= \left (0, \ldots, \underbrace{\varphi_n}_{\mbox{$j$-th column}}, \ldots, 0 \right ),\enspace 1\leq n\leq\omega,\enspace 1\leq j\leq d.
$$
	Then
	$$
T_{\Ub_0} \Phi_{n,j} = v_n \Phi_{n,j}, \enspace 1\leq n\leq\omega,\enspace 1\leq j\leq d.
$$
	Furthermore, $(\Phi_{n,j})_{1\leq n\leq\omega,1\leq j\leq d}$ forms an orthonormal basis of $\mathfrak{H}_1(M_d(\mathbb{C}))$.
\end{remark}

\section{Weak convergence}\label{section weak}
Let $\Ub_0=\Ub_\infty+\mathbf V_0$ satisfy \eqref{stronginitial 1} and \eqref{stronginitial 2} with $s=1/2$. Assume that $T_{\Ub_0}$ on $L^2_+(\R;\C^d)$ has simple discrete spectrum $(v_n)_{1\leq n\leq\omega}$ and
\begin{align}\label{infty sum bound}
 \sum_{n=1}^{\omega}\sqrt{1-v_n^2}<\infty.
\end{align}
The next lemma gives $\mathbf V_0\in C_0(\R;M_d(\C))$, so the domain results and distorted Plancherel formula of Section \ref{section eigenfunctions} apply.
\begin{lemma}\label{spectral Linfty}
Let $\Ub=\Ub_\infty+\mathbf V\in H^{1/2}_\bullet(\R;\Gr_k(\C^d))$, and let $\{v_n\}_{1\leq n\leq\omega}$ be the discrete eigenvalues of $T_\Ub$ on $L^2_+(\R;\C^d)$, counted with multiplicity. If \eqref{infty sum bound} holds, then $\mathbf V$ has a representative in $C_0(\R;M_d(\C))$, and
$$
 \|\mathbf V\|_{L^\infty}\lesssim_d\|\mathbf V\|_{\dot B^1_{1,1}}\lesssim_d\|H_\Ub\|_{\mathfrak S_1}=\sum_{n=1}^{\omega}\sqrt{1-v_n^2}.
$$
\end{lemma}
\begin{proof}
By \eqref{TandK}, the positive eigenvalues of $K_\Ub=H_\Ub^*H_\Ub$ are $1-v_n^2$, counted with multiplicity. Hence
$$
\|H_{\mathbf V}\|_{\mathfrak S_1}=\|H_\Ub\|_{\mathfrak S_1}=\sum_{n=1}^{\omega}\sqrt{1-v_n^2}<\infty.
$$
Peller's theorem \cite{peller2003hankel} and the decomposition $L^2=L^2_+\oplus L^2_-$ \cite{lenzmann2018short} give
\begin{align}\label{fourAtBS1}
 \|\mathbf V\|_{\dot B^1_{1,1}}\lesssim\|[\mathcal H,\mathbf V]\|_{\mathfrak S_1}\lesssim\|H_{\mathbf V}\|_{\mathfrak S_1}.
\end{align}
Let $(\Delta_j)_{j\in\mathbb Z}$ be a homogeneous Littlewood--Paley decomposition, with low-frequency cut-off $S_j$. Bernstein's inequality and \eqref{fourAtBS1} yield
\begin{align}\label{fourBernstein1}
 \sum_{j\in\mathbb Z}\|\Delta_j\mathbf V\|_{L^\infty}\lesssim\sum_{j\in\mathbb Z}2^j\|\Delta_j\mathbf V\|_{L^1}<\infty.
\end{align}
Since each $\Delta_j\mathbf V$ belongs to $C_0(\R;M_d(\C))$, the series converges uniformly in this space. Moreover,
\begin{align}\label{fourBernstein2}
 \|S_j\mathbf V\|_{L^\infty}\lesssim2^{j/2}\|\mathbf V\|_{L^2}\to0\enspace\text{as }j\to-\infty.
\end{align}
The distributional identity $\mathbf V-S_j\mathbf V=\sum_{k\geq j}\Delta_k\mathbf V$ and \eqref{fourBernstein1}--\eqref{fourBernstein2} identify the uniform sum with $\mathbf V$. Thus $\mathbf V\in C_0(\R;M_d(\C))$, and the same estimates give $\|\mathbf V\|_{L^\infty}\lesssim\|\mathbf V\|_{\dot B^1_{1,1}}$; see also \cite[Proposition 2.39]{bahouri2011fourier}.
\end{proof}
The explicit formula of Proposition \ref{global solution matrix} reads $\Pi_+\mathbf V(t,z)=\Omega_t(\Pi_+\mathbf V_0)(z)$, where
$$
 \Omega_tF(z):=\frac{1}{2\pi\ii}I_+\big((X^*+tT_{\Ub_0}-z\id)^{-1}F\big),\qquad z\in\C_+.
$$
\begin{lemma}\label{Omega jie}
For $t\geq0$ and $F\in L^2_+(\R;M_d(\C))$, one has $\Omega_tF\in L^2_+(\R;M_d(\C))$, $\|\Omega_tF\|_{L^2}=\|F\|_{L^2}$, and
$$
 |\Omega_tF(z)|_{\mathcal V}\leq\frac{\|F\|_{L^2}}{2\sqrt{\pi\operatorname{Im}z}},\qquad z\in\C_+.
$$
\end{lemma}
\begin{proof}
Proposition \ref{global solution matrix} gives the isometry $\Omega_t=\mathcal O_{tT_{\Ub_0}}$. Fourier inversion, Cauchy--Schwarz, and Plancherel yield
$$
 |\Omega_tF(z)|_{\mathcal V}\leq\frac1{2\pi}\left(\int_0^\infty e^{-2\xi\operatorname{Im}z}\,\mathrm d\xi\right)^{1/2}\|\widehat{\Omega_tF}\|_{L^2}=\frac{\|F\|_{L^2}}{2\sqrt{\pi\operatorname{Im}z}}.
$$
\end{proof}
\begin{lemma}\label{tozero}
If $F\in L^2_+(\R;M_d(\C))$ satisfies $F\perp\Phi_{n,j}$ for some $1\leq n\leq\omega$ and all $1\leq j\leq d$, then $\Omega_tF(z+tv_n)\to0$ as $t\to+\infty$ for every $z\in\C_+$.
\end{lemma}
\begin{proof}
The functional calculus of $T_{\Ub_0}$ gives $H\in L^2_+(\R;M_d(\C))$ satisfying $(T_{\Ub_0}-v_n\id)H=F$. Thus
$$
 \widehat H(\xi)+(v_n\id-\Ub_\infty)^{-1}\widehat F(\xi)=(v_n\id-\Ub_\infty)^{-1}\widehat{\mathbf V_0H}(\xi),\qquad\xi>0.
$$
If $F\in\dom(X^*)\cap\{\Phi_{n,j}\}_{j=1}^d{}^\perp$, the argument of Proposition \ref{varphiin} gives $H\in\dom(X^*)$. The resolvent identity then yields
\begin{align*}
 \Omega_tF(z+tv_n)
 &=\frac{1}{2\pi\ii}I_+\left(\frac Ht-\frac1t(X^*+t(T_{\Ub_0}-v_n\id)-z\id)^{-1}(X^*H-zH)\right)\\
 &=\frac{I_+(H)}{2\pi\ii t}-\frac1t\Omega_t(X^*H-zH)(z+tv_n).
\end{align*}
Lemma \ref{Omega jie} bounds the last transform by $\|X^*H-zH\|_{L^2}/(2\sqrt{\pi\operatorname{Im}z})$, proving the limit. Since $\dom(X^*)\cap\{\Phi_{n,j}\}_{j=1}^d{}^\perp$ is dense in $\{\Phi_{n,j}\}_{j=1}^d{}^\perp$, the same bound extends the limit to every such $F$.
\end{proof}
Lemma \ref{Omega jie} gives a uniform $L^2$ bound, so the pointwise limits below also determine the weak limits.
\begin{lemma}\label{PiVlim}
For $1\leq n\leq\omega$,
$$
 \Pi_+\mathbf V(t,z+tv_n)\xrightarrow{t\to\pm\infty}\frac{A_n^\infty}{z-y_n^\infty+\ii\delta_n^\infty},\qquad z\in\C_+,
$$
where
\begin{align*}
 y_n^\infty&=\operatorname{Re}\langle X^*\varphi_n,\varphi_n\rangle_{L^2},\qquad
 \delta_n^\infty=-\operatorname{Im}\langle X^*\varphi_n,\varphi_n\rangle_{L^2}=\frac{|I_+(\varphi_n)|_{\mathcal V}^2}{4\pi}>0,\\
 A_n^\infty&=\frac{\ii}{2\pi}\sum_{j=1}^d\langle\Pi_+\mathbf V_0,\Phi_{n,j}\rangle_{L^2}I_+(\Phi_{n,j}).
\end{align*}
\end{lemma}
\begin{proof}
The proof of \cite[Lemma 7.1]{gerard2025global} gives $-\operatorname{Im}\langle X^*\varphi_n,\varphi_n\rangle_{L^2}=|I_+(\varphi_n)|_{\mathcal V}^2/(4\pi)>0$. Set
$$
 F_n:=\Pi_+\mathbf V_0-\sum_{j=1}^d\langle\Pi_+\mathbf V_0,\Phi_{n,j}\rangle_{L^2}\Phi_{n,j}.
$$
Then $F_n\perp\Phi_{n,j}$ for all $j$, so Lemma \ref{tozero} gives $\Omega_tF_n(z+tv_n)\to0$ as $t\to+\infty$. Write $a_n=\langle X^*\varphi_n,\varphi_n\rangle_{L^2}$ and set
$$
 \Psi_{n,j}:=\frac{\Phi_{n,j}}{a_n-z},\qquad
 G_{n,j}:=\frac{X^*\Phi_{n,j}-a_n\Phi_{n,j}}{a_n-z},\qquad 1\leq j\leq d.
$$
Since $(X^*+t(T_{\Ub_0}-v_n\id)-z\id)\Psi_{n,j}=G_{n,j}+\Phi_{n,j}$,
$$
 \Omega_t\Phi_{n,j}(z+tv_n)=\frac{I_+(\Psi_{n,j})}{2\pi\ii}-\Omega_tG_{n,j}(z+tv_n).
$$
Direct calculation gives $G_{n,j}\perp\Phi_{n,k}$ for every $k$. Lemma \ref{tozero} therefore gives
$$
 \Pi_+\mathbf V(t,z+tv_n)\to\frac{\ii}{2\pi}\sum_{j=1}^d\frac{\langle\Pi_+\mathbf V_0,\Phi_{n,j}\rangle_{L^2}I_+(\Phi_{n,j})}{z-a_n}=\frac{A_n^\infty}{z-y_n^\infty+\ii\delta_n^\infty}.
$$
The map $\Ub^-(t,x):=-\Ub(-t,x)$ solves \eqref{HWMm} in $\Gr_{d-k}(\C^d)$ with initial data $-\Ub_0$. Its Toeplitz operator is $-T_{\Ub_0}$, with eigenfunction $\varphi_n$ associated with $-v_n$, and
$$
 y_n^{-,\infty}=y_n^\infty,\qquad\delta_n^{-,\infty}=\delta_n^\infty,\qquad A_n^{-,\infty}=-A_n^\infty.
$$
Writing $s=-t\to+\infty$ gives
$$
 \Pi_+\mathbf V(t,z+tv_n)=-\Pi_+\mathbf V^-(s,z-sv_n)\to-\frac{A_n^{-,\infty}}{z-y_n^{-,\infty}+\ii\delta_n^{-,\infty}}=\frac{A_n^\infty}{z-y_n^\infty+\ii\delta_n^\infty}.
$$
\end{proof}
The Toeplitz eigenvalue equation identifies these limit profiles as solitary waves.
\begin{lemma}\label{is soliton}
For $1\leq n\leq\omega$, set
$$
 h_n:=I_+(\varphi_n),\qquad\alpha_{n,j}:=\langle\Pi_+\mathbf V_0,\Phi_{n,j}\rangle_{L^2},\qquad
 \eta_n:=\frac1{2\pi\ii}\sum_{j=1}^d\overline{\alpha_{n,j}}e_j.
$$
Then $h_n\neq0$, $\delta_n^\infty=|h_n|_{\mathcal V}^2/(4\pi)>0$, $A_n^\infty=h_n\eta_n^*$, and
$$
 A_n^\infty\neq0,\qquad\rank(A_n^\infty)=1,\qquad(A_n^\infty)^2=0.
$$
Moreover,
\begin{align}\label{AU UA}
 A_n^\infty\left(\Ub_\infty+\frac{\ii A_n^{\infty *}}{2\delta_n^\infty}\right)
 +\left(\Ub_\infty+\frac{\ii A_n^{\infty *}}{2\delta_n^\infty}\right)A_n^\infty=0.
\end{align}
Consequently,
$$
 \Qv_{v_n}(x):=\Ub_\infty+\frac{A_n^\infty}{x-y_n^\infty+\ii\delta_n^\infty}+\frac{A_n^{\infty *}}{x-y_n^\infty-\ii\delta_n^\infty}
$$
belongs to $\mathcal{R}at(\R;\Gr_k(\C^d))$ and satisfies
\begin{align}\label{s w e}
 -v_n\Qv_{v_n}'=-\frac{\ii}{2}[\Qv_{v_n},|D|\Qv_{v_n}].
\end{align}
Thus $\Qv_{v_n}(x-v_nt)$ is a traveling solitary wave of \eqref{HWMm}.
\end{lemma}
\begin{proof}
Write $v=v_n$, $\varphi=\varphi_n$, $h=h_n$, $\delta=\delta_n^\infty$, $A=A_n^\infty$, $\eta=\eta_n$, and $y=y_n^\infty$. Lemma \ref{PiVlim} gives $h\neq0$, $\delta=|h|_{\mathcal V}^2/(4\pi)>0$, and
$$
 A=\frac{\ii}{2\pi}\sum_{j=1}^d\alpha_{n,j}he_j^*=h\eta^*.
$$
Since $\alpha_{n,j}=\int_\R\Tr(\mathbf V_0\Phi_{n,j}^*)\,\mathrm dx=\int_\R\varphi^*\mathbf V_0e_j\,\mathrm dx$, one has
\begin{align}\label{b and eta}
 2\pi\ii\eta=\sum_{j=1}^d\overline{\alpha_{n,j}}e_j=\int_\R\mathbf V_0(x)\varphi(x)\,\mathrm dx=:b.
\end{align}
Equation \eqref{lianxu hds} gives $b=vh-\Ub_\infty h$, hence
\begin{align}\label{Uhidentity}
 \Ub_\infty h=vh-2\pi\ii\eta.
\end{align}
The proof of Corollary \ref{hds} gives $\langle h,b\rangle_{\mathcal V}=0$, so
\begin{align}\label{heta orthogonality}
 \langle h,\eta\rangle_{\mathcal V}=0.
\end{align}
Applying $\Ub_\infty$ to \eqref{Uhidentity} yields
\begin{align}\label{Uetaidentity}
 \Ub_\infty\eta=-v\eta-\frac{1-v^2}{2\pi\ii}h.
\end{align}
By \eqref{b and eta} and Corollary \ref{hds},
\begin{align}\label{eta m2}
 |\eta|_{\mathcal V}^2=\frac{|b|_{\mathcal V}^2}{4\pi^2}=\frac{(1-v^2)|h|_{\mathcal V}^2}{4\pi^2}=\frac{(1-v^2)\delta}{\pi}>0.
\end{align}
Thus $A=h\eta^*\neq0$, $\rank(A)=1$, and \eqref{heta orthogonality} gives $A^2=h\eta^*h\eta^*=0$ and $\Tr(A)=\eta^*h=0$. Equations \eqref{Uhidentity}, \eqref{Uetaidentity}, and \eqref{eta m2} give
\begin{align*}
 \Ub_\infty A+A\Ub_\infty
 &=(vh-2\pi\ii\eta)\eta^*+h\left(-v\eta-\frac{1-v^2}{2\pi\ii}h\right)^*\\
 &=-2\pi\ii\eta\eta^*+\frac{1-v^2}{2\pi\ii}hh^*=\frac{|\eta|_{\mathcal V}^2hh^*+|h|_{\mathcal V}^2\eta\eta^*}{2\ii\delta}
 =\frac{AA^*+A^*A}{2\ii\delta},
\end{align*}
which is \eqref{AU UA}.

Set $p=x-y+\ii\delta$, $q=x-y-\ii\delta$, and $S=AA^*+A^*A$. Using $1/(pq)=(1/q-1/p)/(2\ii\delta)$, $A^2=0$, and \eqref{AU UA}, we obtain
$$
\Qv_v^2-\mathbbm{1}_d
 =\frac{A^2}{p^2}+\frac{(A^*)^2}{q^2}+\frac{\Ub_\infty A+A\Ub_\infty-S/(2\ii\delta)}p
 +\frac{\Ub_\infty A^*+A^*\Ub_\infty+S/(2\ii\delta)}q=0.
$$
Since $\Qv_v=\Qv_v^*$ and $\Tr(\Qv_v)=\Tr(\Ub_\infty)=d-2k$, it follows that $\Qv_v(x)\in\Gr_k(\C^d)$.

To verify \eqref{s w e}, first note that \eqref{Uhidentity}, \eqref{Uetaidentity}, and \eqref{eta m2} imply
$$
 [\Ub_\infty,A]=2vA+\frac{\ii}{2\delta}[A,A^*],\qquad
 [\Ub_\infty,A^*]=-2vA^*+\frac{\ii}{2\delta}[A,A^*].
$$
Since $p^{-1}\in L^2_+(\R)$, $q^{-1}\in L^2_-(\R)$, and $D=-\ii\partial_x$,
$$
 |D|\Qv_v=\frac{\ii A}{p^2}-\frac{\ii A^*}{q^2},\qquad
 \Qv_v'=-\frac A{p^2}-\frac{A^*}{q^2}.
$$
The identity $1/(pq^2)+1/(qp^2)=(1/q^2-1/p^2)/(2\ii\delta)$ then gives
\begin{align*}
 [\Qv_v,|D|\Qv_v]
 &=\frac{\ii[\Ub_\infty,A]}{p^2}-\frac{\ii[\Ub_\infty,A^*]}{q^2}-\ii[A,A^*]\left(\frac1{pq^2}+\frac1{qp^2}\right)\\
 &=\frac{2\ii vA}{p^2}+\frac{2\ii vA^*}{q^2}=-2\ii v\Qv_v',
\end{align*}
proving \eqref{s w e}.
\end{proof}

\begin{remark}
Write $\Qv_{v_n}=\Ub_\infty+\Rb_{v_n}$, where
$$
 \Rb_{v_n}(x)=\frac{A_n^\infty}{x-y_n^\infty+\ii\delta_n^\infty}+\frac{A_n^{\infty *}}{x-y_n^\infty-\ii\delta_n^\infty}.
$$
The Pythagorean theorem and \eqref{identity} give $|A_n^\infty|_{\mathcal V}=\sqrt{1-v_n^2}|I_+(\varphi_n)|_{\mathcal V}^2/(2\pi)=2\sqrt{1-v_n^2}\,\delta_n^\infty$. Hence
\begin{align}\label{Rvn Linfty bound}
 \|\Rb_{v_n}\|_{L^\infty}\leq\frac{2|A_n^\infty|_{\mathcal V}}{\delta_n^\infty}=4\sqrt{1-v_n^2}.
\end{align}
The formula
$$
 \|\Pi_+\Rb_{v_n}\|_{\dot H^s}^2=2\pi\Gamma(2s+1)\frac{|A_n^\infty|_{\mathcal V}^2}{(2\delta_n^\infty)^{2s+1}}
$$
yields $\|\Rb_{v_n}\|_{L^2}=\sqrt{2}\|\Pi_+\Rb_{v_n}\|_{L^2}=\sqrt{2}\sqrt{1-v_n^2}|I_+(\varphi_n)|_{\mathcal V}$ and
\begin{align}\label{RvnH12}
 \|\Rb_{v_n}\|_{\dot H^{1/2}}=\sqrt{2}\|\Pi_+\Rb_{v_n}\|_{\dot H^{1/2}}=2\sqrt{\pi}\sqrt{1-v_n^2}.
\end{align}
\end{remark}

\section{Proof of Theorem \ref{infinitysolitons}: Soliton resolution in the $\dot{H}^{1/2}$ seminorm}\label{p of 1}
We prove the following result for \eqref{HWMm}.
\begin{proposition}[Soliton resolution for \eqref{HWMm} in the $\dot{H}^{1/2}$ seminorm]\label{infinitysolitons general}
Let $\Ub_0=\Ub_\infty+\mathbf V_0\in H^{1/2}_\bullet(\R;\Gr_k(\C^d))$ be non-rational. Suppose that the discrete spectrum $(v_n)_{1\leq n\leq\infty}$ of $T_{\Ub_0}$ on $L^2_+(\R;\C^d)$ is simple and satisfies \eqref{infty sum bound}. Then there exist ground state solitary wave profiles $(\mathbf Q_{v_n})_{1\leq n\leq\infty}$ such that
$$
\lim_{m\to\infty}\limsup_{t\to\pm\infty}\left\|\mathbf V(t)-\sum_{n=1}^m\left(\mathbf Q_{v_n}(\cdot-v_nt)-\Ub_\infty\right)\right\|_{H^{1/2}}=0
$$
and
$$
\lim_{t\to\pm\infty}\left\|\mathbf V(t)-\sum_{n=1}^\infty\left(\mathbf Q_{v_n}(\cdot-v_nt)-\Ub_\infty\right)\right\|_{\dot H^{1/2}}=0.
$$
\end{proposition}
\begin{proof}
Set
\begin{align}\label{guzijishu}
\Ub_\text{sol}^\infty(t,x):=\Ub_\infty+\sum_{n=1}^\infty\Rb_{v_n}(x-v_nt).
\end{align}
The spectral assumption gives
\begin{align}\label{series singular bound}
\sum_{n=1}^\infty\|\Rb_{v_n}\|_{\dot H^{1/2}}=2\sqrt\pi\sum_{n=1}^\infty\sqrt{1-v_n^2}<\infty,
\end{align}
so \eqref{guzijishu} converges in $\dot H^{1/2}(\R;M_d(\C))$, uniformly in $t$. For $m\geq1$, set
$$
\Ub_\text{sol}^m(t):=\Ub_\infty+\sum_{n=1}^m\Rb_{v_n}(\cdot-v_nt),\qquad
\mathbf r_m(t):=\Ub(t)-\Ub_\text{sol}^m(t).
$$
Write $f(t)=\Pi_+\mathbf V(t)$ and $g_n(t)=\Pi_+\Rb_{v_n}(\cdot-v_nt)$. For $X=L^2$ or $X=\dot H^{1/2}$, expansion of the squared norm gives
\begin{align*}
\|\Pi_+\mathbf r_m(t)\|_X^2
&=\|f(t)\|_X^2+\sum_{n=1}^m\|g_n(t)\|_X^2-2\operatorname{Re}\sum_{n=1}^m\langle f(t),g_n(t)\rangle_X\\
&\enspace+\sum_{\substack{1\leq j,k\leq m\\j\neq k}}\langle g_j(t),g_k(t)\rangle_X.
\end{align*}
The explicit formula, Proposition \ref{Plancherel}, and Lemma \ref{Omega jie} give
$$
\|f(t)\|_{L^2}^2=\|\Omega_t\Pi_+\mathbf V_0\|_{L^2}^2=\|\Pi_+\mathbf V_0\|_{L^2}^2=\sum_{n=1}^\infty\|\Pi_+\Rb_{v_n}\|_{L^2}^2.
$$
Energy conservation and \cite[Lemma 4.1]{gerard2025global} give
$$
\|f(t)\|_{\dot H^{1/2}}^2=2\pi\Tr K_{\Ub_0}=2\pi\sum_{n=1}^\infty(1-v_n^2)=\sum_{n=1}^\infty\|\Pi_+\Rb_{v_n}\|_{\dot H^{1/2}}^2.
$$
By Lemma \ref{PiVlim}, for each $n$,
$$
\langle f(t),g_n(t)\rangle_X\longrightarrow\|\Pi_+\Rb_{v_n}\|_X^2\qquad(t\to\pm\infty).
$$
For $X=\dot H^{1/2}$, this follows by writing the pairing as $\langle\Pi_+\mathbf V(t,\cdot+v_nt),|D|\Pi_+\Rb_{v_n}\rangle_{L^2}$. Distinct velocities also imply $\langle g_j(t),g_k(t)\rangle_X\to0$ for $j\neq k$. Hence
$$
\lim_{t\to\pm\infty}\|\Pi_+\mathbf r_m(t)\|_X^2=\sum_{n=m+1}^\infty\|\Pi_+\Rb_{v_n}\|_X^2\longrightarrow0\qquad(m\to\infty).
$$
Since $\mathbf r_m$ is Hermitian, this proves the assertion in $H^{1/2}$. Moreover,
$$
\limsup_{t\to\pm\infty}\|\Ub(t)-\Ub_\text{sol}^\infty(t)\|_{\dot H^{1/2}}
\leq\left(2\sum_{n=m+1}^\infty\|\Pi_+\Rb_{v_n}\|_{\dot H^{1/2}}^2\right)^{1/2}
+\sum_{n=m+1}^\infty\|\Rb_{v_n}\|_{\dot H^{1/2}}.
$$
Letting $m\to\infty$ and using \eqref{series singular bound} proves the remaining assertion.
\end{proof}
Theorem \ref{infinitysolitons} follows by taking $d=2$, $k=1$ in \eqref{HWMm}.

\section{Proof of Theorem \ref{infinitysolitonsLinfty}: Soliton resolution in the $L^\infty$ norm}\label{p of 2}
We prove the following result for \eqref{HWMm}.
\begin{proposition}[Soliton resolution for \eqref{HWMm} in the $L^\infty$ norm]\label{infinitysolitonsLinfty general}
Under the assumptions and notation of Proposition \ref{infinitysolitons general}, the series $\sum_{n=1}^\infty(\mathbf Q_{v_n}(x-v_nt)-\Ub_\infty)$ converges absolutely and uniformly in $x\in\R$, uniformly in $t\in\R$, and
$$
\lim_{t\to\pm\infty}\left\|\mathbf V(t)-\sum_{n=1}^\infty\left(\mathbf Q_{v_n}(\cdot-v_nt)-\Ub_\infty\right)\right\|_{L^\infty}=0.
$$
\end{proposition}
Set $\mathbf S(t):=\sum_{n=1}^\infty\Rb_{v_n}(\cdot-v_nt)$. By \eqref{Rvn Linfty bound} and \eqref{series singular bound}, this series converges absolutely and uniformly in $x$, uniformly in $t$. Each summand belongs to $C_0(\R;M_d(\C))$, so $\mathbf S(t)\in C_0(\R;M_d(\C))$. Define $A(t):=H_{\Ub(t)}=H_{\mathbf V(t)}$ and $B(t):=H_{\mathbf S(t)}$.
\begin{lemma}\label{Hankel singular values}
The operators $A(t),B(t)$ belong to $\mathfrak S_1$. The singular values of $A(t)$ are $a_n:=\sqrt{1-v_n^2}$, ordered as $a_1\geq a_2\geq\cdots>0$. Writing $H_n(t):=H_{\Rb_{v_n}(\cdot-v_nt)}$, one has $B(t)=\sum_{n=1}^\infty H_n(t)$ in $\mathfrak S_1$, and $a_n$ is the unique nonzero singular value of $H_n(t)$.
\end{lemma}
\begin{proof}
By \eqref{TandK}, preservation of the discrete spectrum, and \eqref{KandH}, the nonzero eigenvalues of $K_{\Ub(t)}=A(t)^*A(t)$ are $1-v_n^2$. Thus $A(t)$ has singular values $a_n$ and belongs to $\mathfrak S_1$ by \eqref{series singular bound}.

For each $n$, direct calculation gives
$$
(H_n(t)f)(x)=\frac{A_n^{\infty *}f(v_nt+y_n^\infty+\ii\delta_n^\infty)}{x-v_nt-y_n^\infty-\ii\delta_n^\infty},\qquad f\in L^2_+(\R;\C^d).
$$
Lemma \ref{is soliton} gives $\rank A_n^\infty=1$, hence $\rank H_n(t)=1$. By \eqref{RvnH12} and \cite[Lemma 4.1]{gerard2025global},
$$
s_1(H_n(t))^2=\Tr(H_n(t)^*H_n(t))=\frac1{2\pi}\|\Pi_+\Rb_{v_n}\|_{\dot H^{1/2}}^2=1-v_n^2=a_n^2.
$$
Consequently, $\sum_n H_n(t)$ converges in $\mathfrak S_1$. Its limit equals $B(t)$ because
$$
\left\|B(t)-\sum_{n=1}^N H_n(t)\right\|_{\mathrm{op}}\leq\sum_{n=N+1}^\infty\|\Rb_{v_n}\|_{L^\infty}\longrightarrow0\qquad(N\to\infty).
$$
\end{proof}
Lemma \ref{spectral Linfty} and Proposition \ref{global solution matrix} give $\mathbf V(t)\in C_0(\R;M_d(\C))$ and $\|\mathbf V(t)\|_{\dot B^1_{1,1}}\lesssim\sum_n a_n$. Therefore,
\begin{align}\label{FtC0}
\mathbf F(t):=\mathbf V(t)-\mathbf S(t)\in C_0(\R;M_d(\C)).
\end{align}
Set $C(t):=H_{\mathbf F(t)}=A(t)-B(t)\in\mathfrak S_1$. By \eqref{series singular bound}, $\mathbf F(t)\in\dot H^{1/2}(\R;M_d(\C))$. Since $\mathbf F(t)=\mathbf F(t)^*$, \cite[Lemma 4.1]{gerard2025global} gives
\begin{align}\label{CS2H}
\|C(t)\|_{\mathfrak S_2}\sim\|\mathbf F(t)\|_{\dot H^{1/2}}.
\end{align}
\begin{lemma}\label{tail trace}
For every $N\geq1$,
\begin{align}\label{tail trace norm estimate}
\|C(t)\|_{\mathfrak S_1}\leq\sqrt{2N}\,\|C(t)\|_{\mathfrak S_2}+2\sum_{n=N+1}^\infty a_n.
\end{align}
\end{lemma}
\begin{proof}
Let $A_N(t)$ be the truncation of $A(t)$ to its first $N$ singular values, $B_N(t):=\sum_{n=1}^N H_n(t)$, and $D_N(t):=A_N(t)-B_N(t)$. Lemma \ref{Hankel singular values} gives $\rank D_N(t)\leq2N$ and
$$
\|A(t)-A_N(t)\|_{\mathfrak S_1}=\sum_{n>N}a_n,\qquad
\|B(t)-B_N(t)\|_{\mathfrak S_1}\leq\sum_{n>N}a_n.
$$
Thus $\|C(t)-D_N(t)\|_{\mathfrak S_1}\leq2\sum_{n>N}a_n$. Fan's inequality \cite{fan1951maximum} (see also \cite[Theorem 1.7]{simon2005trace}) yields $s_{2N+j}(C(t))\leq s_j(C(t)-D_N(t))$ for $j\geq1$, so
$$
\sum_{n>2N}s_n(C(t))\leq\|C(t)-D_N(t)\|_{\mathfrak S_1}\leq2\sum_{n>N}a_n.
$$
Applying Cauchy--Schwarz to the first $2N$ singular values proves \eqref{tail trace norm estimate}.
\end{proof}
Peller's theorem \cite{peller2003hankel} and the orthogonal decomposition $L^2=L^2_+\oplus L^2_-$ \cite{lenzmann2018short} give
\begin{align}\label{FtBS1}
\|\mathbf F(t)\|_{\dot B^1_{1,1}}\sim\|[\mathcal H,\mathbf F(t)]\|_{\mathfrak S_1}\sim\|C(t)\|_{\mathfrak S_1}.
\end{align}
By \eqref{FtC0} and \cite[Proposition 2.39]{bahouri2011fourier},
\begin{align}\label{FtLinftyB}
\|\mathbf F(t)\|_{L^\infty}\lesssim\|\mathbf F(t)\|_{\dot B^1_{1,1}}.
\end{align}
Combining \eqref{CS2H}, \eqref{tail trace norm estimate}, \eqref{FtBS1}, and \eqref{FtLinftyB} gives
$$
\|\mathbf F(t)\|_{L^\infty}\lesssim\sqrt N\,\|\mathbf F(t)\|_{\dot H^{1/2}}+\sum_{n>N}a_n.
$$
First let $t\to\pm\infty$ using Proposition \ref{infinitysolitons general}, then let $N\to\infty$. This proves Proposition \ref{infinitysolitonsLinfty general}. Taking $d=2$, $k=1$ gives Theorem \ref{infinitysolitonsLinfty}.

\section{Proof of Theorem \ref{finite to infinite}: Examples of the admissible non-rational data}\label{p of 3}
In this section, we set $\mathcal V=\C^2$.
\begin{lemma}\label{one pole extension}
Let $\mathbf e_3=(0,0,1)^\top$ and
$$
\uu_{*,0}=\mathbf e_3+\mathbf v_{*,0}\in\mathcal{R}at(\R;\Ss^2),\enspace \Ub_{*,0}=\uu_{*,0}\cdot\bm\sigma.
$$
Suppose that $\rank(K_{\Ub_{*,0}})=N\geq1$ and all positive eigenvalues of $K_{\Ub_{*,0}}$ are simple. For every $m\in\N^*$ and $\varepsilon>0$, there exists
$$
\widetilde\uu_{*,0}=\mathbf e_3+\widetilde{\mathbf v}_{*,0}\in\mathcal{R}at(\R;\Ss^2),\enspace \widetilde\Ub_{*,0}=\widetilde\uu_{*,0}\cdot\bm\sigma,
$$
such that $K_{\widetilde\Ub_{*,0}}$ has rank $N+1$, all its positive eigenvalues are simple, and
$$
 \|\widetilde{\mathbf v}_{*,0}-\mathbf v_{*,0}\|_{H^m}+\|\widetilde\Ub_{*,0}-\Ub_{*,0}\|_{L^\infty}+\|H_{\widetilde\Ub_{*,0}}-H_{\Ub_{*,0}}\|_{\mathfrak S_1}<\varepsilon.
$$
\end{lemma}
\begin{proof}
By \cite[Theorem 8.1]{gerard2025global}, there are coprime complex polynomials $P,Q$ with $\deg P=N$ and $\deg Q\leq N-1$ such that
$$
 \uu_{*,0}=\mathscr S(P,Q):=\left(\frac{2\operatorname{Re}(P\overline Q)}{|P|^2+|Q|^2},\frac{2\operatorname{Im}(P\overline Q)}{|P|^2+|Q|^2},\frac{|P|^2-|Q|^2}{|P|^2+|Q|^2}\right)^\top.
$$
Choose $\zeta\in\C\setminus\R$ with $Q(\zeta)\neq0$. For $\eta\neq0$, set
$$
 P_\eta=(X-\zeta)P+\eta Q,\enspace Q_\eta=(X-\zeta)Q,\enspace \frac{P_\eta}{Q_\eta}=\frac P Q+\frac\eta{X-\zeta}.
$$
Since $P_\eta(\zeta)=\eta Q(\zeta)\neq0$ and $P_\eta(\alpha)=(\alpha-\zeta)P(\alpha)\neq0$ at every zero $\alpha$ of $Q$, the polynomials $P_\eta,Q_\eta$ are coprime. Their degrees are $N+1$ and at most $N$, so \cite[Theorem 8.1]{gerard2025global} gives
$$
 \uu_\eta:=\mathscr S(P_\eta,Q_\eta)\in\mathcal{R}at(\R;\Ss^2),\enspace \rank(K_{\Ub_\eta})=N+1.
$$
For real $x$, $|x-\zeta|\geq c_\zeta:=|\operatorname{Im}\zeta|$ and $|P(x)|^2+|Q(x)|^2>0$. Set $D_\eta=|P_\eta|^2+|Q_\eta|^2$. Then $D_0(x)\langle x\rangle^{-2N-2}$ has a positive infimum, and
\begin{align*}
 D_\eta-D_0&=2\operatorname{Re}\bigl(\eta Q\overline{(X-\zeta)P}\bigr)+|\eta|^2|Q|^2,\\
 |D_\eta-D_0|&\leq\bigl(|\eta|/c_\zeta+|\eta|^2/c_\zeta^2\bigr)D_0.
\end{align*}
Hence $D_\eta\geq D_0/2$ uniformly in $x$ for small $|\eta|$. Differentiating the rational expressions gives, for every $m\in\N^*$,
$$
 \max_{0\leq k\leq m}|\partial_x^k(\uu_\eta-\uu_{*,0})(x)|\leq C_m|\eta|\langle x\rangle^{-1}.
$$
Thus $\|\uu_\eta-\uu_{*,0}\|_{H^m}+\|\Ub_\eta-\Ub_{*,0}\|_{L^\infty}\to0$. Since
$$
 \rank(H_{\Ub_\eta}-H_{\Ub_{*,0}})\leq\rank H_{\Ub_\eta}+\rank H_{\Ub_{*,0}}=2N+1,
$$
we also have $\|H_{\Ub_\eta}-H_{\Ub_{*,0}}\|_{\mathfrak S_1}\leq(2N+1)\|\Ub_\eta-\Ub_{*,0}\|_{L^\infty}\to0$.

Let $\mu_j(K)$ be the decreasing eigenvalues of a positive compact operator, counted with multiplicity and completed by zeros. Set
\begin{align*}
 \Gamma&:=\min_{1\leq j\leq N}\bigl(\mu_j(K_{\Ub_{*,0}})-\mu_{j+1}(K_{\Ub_{*,0}})\bigr)>0,\\
 d_\eta&:=\|K_{\Ub_\eta}-K_{\Ub_{*,0}}\|_{\mathrm{op}}\leq2\|H_{\Ub_\eta}-H_{\Ub_{*,0}}\|_{\mathrm{op}}\to0.
\end{align*}
The min-max principle gives $|\mu_j(K_{\Ub_\eta})-\mu_j(K_{\Ub_{*,0}})|\leq d_\eta$ for $j\geq1$. For $d_\eta<\Gamma/4$,
$$
 \mu_j(K_{\Ub_\eta})-\mu_{j+1}(K_{\Ub_\eta})\geq\Gamma-2d_\eta>0,\enspace 1\leq j\leq N.
$$
Positivity and rank $N+1$ give $\mu_{N+1}(K_{\Ub_\eta})>0$, proving simplicity.
\end{proof}
\begin{lemma}\label{finite stability}
Let $m\geq2$ be an integer, $T>0$, and $\uu_*=\mathbf e_3+\mathbf v_*$ the global solution of \eqref{HWMv} with rational initial data $\uu_{*,0}=\mathbf e_3+\mathbf v_{*,0}$. For every $\varepsilon>0$, there is $\delta>0$ such that the global rational solution $\widetilde\uu_*=\mathbf e_3+\widetilde{\mathbf v}_*$ with initial data
$$
 \widetilde\uu_{*,0}=\mathbf e_3+\widetilde{\mathbf v}_{*,0}\in\mathcal{R}at(\R;\Ss^2),\enspace \|\widetilde{\mathbf v}_{*,0}-\mathbf v_{*,0}\|_{H^m}<\delta
$$
satisfies $\sup_{|t|\leq T}\|\widetilde\uu_*(t)-\uu_*(t)\|_{H^m}<\varepsilon$.
\end{lemma}
\begin{proof}
Continuity gives $M:=\sup_{|t|\leq T}\|\mathbf v_*(t)\|_{H^m}<\infty$. Choose $R>M+1$ and let $\tau=T(R)>0$ be the uniform local existence time in Proposition \ref{local explicit}. Subdivide $[0,T]$ at $0=t_0<\cdots<t_N=T$ with $t_{j+1}-t_j<\tau$. Continuous dependence at each $\uu_*(t_j)$ propagates sufficiently small $H^m$ differences throughout $[t_j,t_{j+1}]$. Iterating over this finite partition gives $\delta_+>0$ such that
$$
 \|\widetilde{\mathbf v}_{*,0}-\mathbf v_{*,0}\|_{H^m}<\delta_+\ \Longrightarrow\ \sup_{0\leq t\leq T}\|\widetilde\uu_*(t)-\uu_*(t)\|_{H^m}<\varepsilon.
$$
The maps $\widetilde\uu^-_*(t)=-\widetilde\uu_*(-t)$ and $\uu^-_*(t)=-\uu_*(-t)$ are global solutions with rational initial data $-\widetilde\uu_{*,0}$ and $-\uu_{*,0}$. The same argument yields $\delta_->0$ such that
$$
 \|\widetilde{\mathbf v}_{*,0}-\mathbf v_{*,0}\|_{H^m}<\delta_-\ \Longrightarrow\ \sup_{-T\leq t\leq0}\|\widetilde\uu_*(t)-\uu_*(t)\|_{H^m}<\varepsilon.
$$
Take $\delta=\min\{\delta_+,\delta_-\}$.
\end{proof}
\begin{proof}[Proof of Theorem \ref{finite to infinite}]
Rotation invariance of the equation, Sobolev distances, and operator quantities reduces the proof to $\uu_\infty=\mathbf e_3$. In the constant case $\uu_{*,0}\equiv\mathbf e_3$, take $P=X-\ii$ and $Q=a>0$ in the parametrization used in the proof of Lemma \ref{one pole extension}. This gives rational maps $\uu_a$ with $\rank(K_{\uu_a\cdot\bm\sigma})=1$ and $\|\uu_a-\mathbf e_3\|_{H^m}\to0$ as $a\to0$. Choose $a$ so that $\|\uu_a-\mathbf e_3\|_{H^m}<\varepsilon/2$, and apply the construction below to $\uu_a$ with $\varepsilon/2$ in place of $\varepsilon$. The triangle inequality then gives the required approximation to $\mathbf e_3$. We may therefore assume that $\uu_{*,0}$ is nonconstant.

Set
$$
 \uu_0^{(1)}=\uu_{*,0},\enspace \mathbf v_0^{(1)}=\mathbf v_{*,0},\enspace \Ub_0^{(1)}=\Ub_{*,0},\enspace N=\rank(K_{\Ub_0^{(1)}}).
$$
Suppose that $\uu_0^{(n)}=\mathbf e_3+\mathbf v_0^{(n)}$ has been chosen with $\Ub_0^{(n)}=\uu_0^{(n)}\cdot\bm\sigma$ and $K_n:=K_{\Ub_0^{(n)}}$ of rank $N+n-1$, with simple positive eigenvalues $\mu_1^{(n)}>\cdots>\mu_{N+n-1}^{(n)}>0$. Set
$$
 \mu_{N+n}^{(n)}=0,\enspace \Gamma_n:=\min_{1\leq j\leq N+n-1}(\mu_j^{(n)}-\mu_{j+1}^{(n)})>0.
$$
Let $\uu^{(n)}=\mathbf e_3+\mathbf v^{(n)}$ be its global rational solution. Lemma \ref{finite stability} gives $\delta_n>0$ such that every rational perturbation satisfying
$$
 \widetilde\uu_0^{(n)}=\mathbf e_3+\widetilde{\mathbf v}_0^{(n)},\enspace \|\widetilde{\mathbf v}_0^{(n)}-\mathbf v_0^{(n)}\|_{H^{m+n+2}}<\delta_n
$$
has a solution obeying $\sup_{|t|\leq n}\|\widetilde\uu^{(n)}(t)-\uu^{(n)}(t)\|_{H^{m+n+2}}<\varepsilon2^{-n-2}$.
By Lemma \ref{one pole extension}, choose $\uu_0^{(n+1)}$ so that $\rank K_{n+1}=N+n$, its positive eigenvalues are simple, and
\begin{align*}
 &\|\mathbf v_0^{(n+1)}-\mathbf v_0^{(n)}\|_{H^{m+n+2}}+\|H_{\Ub_0^{(n+1)}}-H_{\Ub_0^{(n)}}\|_{\mathfrak S_1}<\min\{\varepsilon2^{-n-2},\delta_n\},\\
 &\|K_{n+1}-K_n\|_{\mathrm{op}}<2^{-n-5}\min_{1\leq r\leq n}\Gamma_r.
\end{align*}
Then
$$
 \sup_{|t|\leq n}\|\mathbf v^{(n+1)}(t)-\mathbf v^{(n)}(t)\|_{H^{m+n+2}}<\varepsilon2^{-n-2}.
$$
The series $\sum_{n\geq1}(\mathbf v_0^{(n+1)}-\mathbf v_0^{(n)})$ converges in every $H^k(\R;\R^3)$, $k\in\N^*$, giving $\uu_0=\mathbf e_3+\mathbf v_0\in H^\infty_\bullet(\R;\Ss^2)$; uniform convergence preserves the sphere constraint. Moreover,
$$
 \|\mathbf v_0-\mathbf v_{*,0}\|_{H^m}\leq\sum_{n=1}^\infty\|\mathbf v_0^{(n+1)}-\mathbf v_0^{(n)}\|_{H^m}<\varepsilon.
$$
The Hankel operators are Cauchy in $\mathfrak S_1$, and uniform convergence of the symbols gives their operator norm limit $H_{\Ub_0}$, where $\Ub_0=\uu_0\cdot\bm\sigma$. Hence $H_{\Ub_0^{(n)}}\to H_{\Ub_0}$ in $\mathfrak S_1$ and $H_{\Ub_0}\in\mathfrak S_1$.

For $K_0:=K_{\Ub_0}$,
$$
 d_n:=\|K_0-K_n\|_{\mathrm{op}}\leq\sum_{r=n}^\infty\|K_{r+1}-K_r\|_{\mathrm{op}}<\Gamma_n\sum_{r=n}^\infty2^{-r-5}<\Gamma_n/16.
$$
The min-max principle gives $|\mu_j(K_0)-\mu_j(K_n)|\leq d_n$ for $j\geq1$, so
$$
 \mu_j(K_0)-\mu_{j+1}(K_0)\geq\mu_j(K_n)-\mu_{j+1}(K_n)-2d_n>\tfrac78\Gamma_n>0,\enspace 1\leq j\leq N+n-1.
$$
Since $n$ is arbitrary, $K_0$ has infinite rank and simple positive eigenvalues $(\lambda_n)_{n\geq1}$. The identity $K_0=H_{\Ub_0}^*H_{\Ub_0}$ and $H_{\Ub_0}\in\mathfrak S_1$ yield $\sum_{n\geq1}\sqrt{\lambda_n}<\infty$.

The series $\sum_{n\geq1}(\mathbf v^{(n+1)}-\mathbf v^{(n)})$ converges in every $C([-T,T];H^k(\R;\R^3))$, $T>0$, $k\in\N^*$, giving
$$
 \uu=\mathbf e_3+\mathbf v\in C(\R;H^\infty_\bullet(\R;\Ss^2)),\enspace \uu(0)=\uu_0.
$$
Here uniform convergence again preserves the sphere constraint. Set $\Ub^{(n)}=\uu^{(n)}\cdot\bm\sigma=\sigma_3+\mathbf V^{(n)}$ and $\Ub=\uu\cdot\bm\sigma=\sigma_3+\mathbf V$. The rational solutions satisfy
\begin{align}\label{rational integral equation}
 \mathbf V^{(n)}(t)=\mathbf V_0^{(n)}-\frac{\ii}{2}\int_0^t\left[\sigma_3+\mathbf V^{(n)}(\tau),|D|\mathbf V^{(n)}(\tau)\right]\,\mathrm d\tau.
\end{align}
For $T>0$ and $k\geq2$, the Sobolev product estimate gives
\begin{align*}
 &\sup_{|t|\leq T}\|[\sigma_3+\mathbf V^{(n)},|D|\mathbf V^{(n)}]-[\sigma_3+\mathbf V,|D|\mathbf V]\|_{H^{k-1}}\\
 &\hspace{2em}\leq C_{T,k}\sup_{|t|\leq T}\|\mathbf V^{(n)}-\mathbf V\|_{H^k}\to0.
\end{align*}
Passing to the limit in \eqref{rational integral equation} yields $\partial_t\Ub=-\frac{\ii}{2}[\Ub,|D|\Ub]$, and hence \eqref{HWMv}.
\end{proof}

\section*{Acknowledgments}\label{s:10} 	
    This work was supported by the National Natural Science Foundation of China under Grant No. 12371255 and the Fundamental Research Funds for the Central Universities of CUMT under Grant No. 2024ZDPYJQ1003.\\

  \textbf{Data availibility}: No data was used for the research described in the article.\\

  \textbf{Declarations}\\

\textbf{Conflict of interest}: The authors declare no conflict of interest.

\bibliographystyle{plain}

\end{document}